\documentclass[a4paper,11pt,final]{article}
\usepackage{a4wide}
\usepackage[latin1]{inputenc}
\usepackage[T1]{fontenc}
\usepackage{amsmath, amsthm}
\usepackage{amssymb}
\usepackage{stmaryrd}
\usepackage{enumerate}
\usepackage{graphicx,psfrag}
\usepackage{epsfig}
\usepackage{mathrsfs}

\usepackage{pgf,pgfarrows,pgfautomata,pgfheaps,pgfnodes,pgfshade}
\usepackage{verbatim}
\DeclareGraphicsExtensions{.jpg,.png,.eps}
\usepackage{graphicx}
\usepackage[latin1]{inputenc}

\theoremstyle{plain}
\newtheorem{proposition}{Proposition}[section]
\newtheorem{theorem}[proposition]{Theorem}

\theoremstyle{definition}
\newtheorem{definition}[proposition]{Definition}

\theoremstyle{remark}

\newtheorem{remark}[proposition]{Remark}

\renewenvironment{proof}{\smallskip\noindent\emph{\textbf{Proof.}}\hspace{1pt}}%
{\hspace{-5pt}{\nobreak\quad\nobreak\hfill\nobreak$\square$\vspace{8pt}%
\par}\smallskip\goodbreak}

\newenvironment{proofof}[1]{\smallskip\noindent\emph{\textbf{Proof of #1.}}%
\hspace{1pt}}{\hspace{-5pt}{\nobreak\quad\nobreak\hfill\nobreak%
$\square$\vspace{8pt}\par}\smallskip\goodbreak}

\newcommand{\lip}{{\rm Lip}}
\newcommand{\Mes}{\mathcal{M}}

\renewcommand{\div}{{\mathrm{div}}}
\newcommand{\Div}{{\mathrm{Div}}}
\renewcommand{\L}[1]{\mathbf{L^#1}}
\newcommand{\Lloc}[1]{\mathbf{L^{#1}_{loc}}}
\newcommand{\C}[1]{\mathscr{C}^{#1}}
\newcommand{\Cc}[1]{\mathscr{C}_c^{#1}}
\newcommand{\modulo}[1]{{\left|#1\right|}}
\newcommand{\norma}[1]{{\left\|#1\right\|}}
\newcommand{\reali}{{\mathbb{R}}}
\newcommand{\naturali}{{\mathbb{N}}}
\newcommand{\tv}{\mathrm{TV}}
\newcommand{\BV}{\mathbf{BV}}
\newcommand{\Lip}{\mathinner\mathbf{Lip}}
\renewcommand{\epsilon}{\varepsilon}
\renewcommand{\phi}{\varphi}
\newcommand{\rpic}{\reali_+}
\newcommand{\rpis}{\reali_+^*}
\newcommand{\pt}{\partial}

\newcommand{\W}[2]{\mathbf{W^{#1,#2}}}
\newcommand{\dst}{\displaystyle}
\renewcommand{\d}[1]{\mathinner{\mathrm{d}{#1}}}
\newcommand{\supp}{\mathrm{Supp}}

\begin{document}

%%-----------------------------
%%      the top matter
%%-----------------------------
\author{Magali {Lécureux-Mercier}$^1$ }

\title{Conservation laws with a non-local flow \\ Application to Pedestrian traffic.}

\footnotetext[1]{Technion, Israel Institute of Technology, Amado Building,
 32000 Haifa, Israel}
%
%\dedicated{\it Dedicated to Maurice Dupont} %if necessary
%

%

%\begin{resume} On présente ici différents modèles de trafic piéton et on prouve l'existence et l'unicité des solutions 
%pour ces modèles. 
%\end{resume}
%
%
\maketitle
%%-----------------------------
%%      your text
%%-----------------------------

\begin{abstract} 
In this note, we introduce some models of pedestrian traffic and prove existence and uniqueness for these models. 
\end{abstract}

\section{Introduction}\label{sec:intro}
In the last decades, the crowds' dynamics has attracted a lot of  scientific interest. A first reason of this interest is to understand  how crowd disasters happened in some panic events, for example at the end of football play, during concerts, in the case of fire, or  in place of pilgrimage (e.g. on Jamarat Bridge in Saudi Arabia, see \cite{HelbingJohanssonZein}). Another 
reason lies in architecture of buildings such as subway stations or stadiums, where a lot of pedestrians are crossing. The goals here are consequently twofold: in one hand we want to understand the behavior of pedestrians in panic and adapt the regulation of traffic in order to avoid deaths; in the other hand, we want to modelize the interaction of several kinds of pedestrians with different objectives and  in particular study how the geometry influences the  general pattern.  

In a macroscopic setting,  a population is described by its density $\rho$ which satisfies the conservation law
\begin{align}
\pt_t \rho +\Div( \rho \, V(t,x,\rho))&=0\,,&\rho(0, \cdot)&=\rho_0\,, \label{eq:nl}
\end{align}
where $V(t, x, \rho)$ is a vector field describing the velocity of the pedestrians depending on the time $t\geq 0$, the space $x\in \reali^N$ and the density $\rho$. According to the choice of $V$,  various behaviors can be observed. Several authors already studied pedestrian traffic in  two dimensions space ($N=2$). Some of these models are local in $\rho$, that is to say $V$ depends on the local density $\rho(t,x)$ \cite{CosciaCanavesio, BellomoDogbe_review, Hughes1, Hughes2, MauryChupin, MauryChupinVenel} ; other models use not only the local density $\rho(t,x)$ but the entire distribution of $\rho$, for example they depend on the convolution product $\rho(t)*\eta$ \cite{DiFrancesco, PiccoliTosin}. Here, in the line of preceding papers \cite{ColomboHertyMercier, ColomboGaravelloMercier, ColomboMercier, ColomboLecureux}, we present nonlocal macroscopic models for pedestrian traffic, we study these models and compare their properties.

Our first aim is to modelize the behavior of pedestrians in different situations: crowd  behaves indeed differently in panic or in a normal situation where courtesy rules do apply. We also introduce models in the case of a population interacting with an individual, and in  the case of  several populations with different objectives. For instance, we want to include in our study the case of two populations crossing in a corridor. 

Second, we want to study the introduced models  and prove existence and uniqueness of solutions under various sets of hypotheses. We will use two kinds of arguments: the first one comes from Kru\v zkov theory \cite{Kruzkov, Lecureux}, the second one from the optimal transport theory. We want to prove existence and uniqueness of solutions for the various models presented below. Let us concentrate on the case of \emph{pedestrians in panic}:
\begin{equation}\label{eq:panic0}
\pt_t \rho+\Div (\rho\, v(\rho*\eta)\vec \nu(x))=0.
\end{equation} 
All the proofs of existence and uniqueness in  this note are based on the following  idea: let us fix the nonlocal term and, instead of (\ref{eq:panic0}), we  study the Cauchy problem
\begin{equation}\label{eq:fix}
\pt_t \rho+\Div (\rho\, v(r*\eta)\vec \nu(x))=0\,, \qquad  \rho(0)=\rho_0\,,
\end{equation}
where $r$ is a given function. Then, we introduce  the application
\begin{equation}\label{eq:Q}
\mathscr{Q}\;:\;\left\{
\begin{array}{ccc}
r & \mapsto & \rho\\
X&\to &X
\end{array}
\right\}\,,
\end{equation} 
where the space $X$ has to be chosen so that
\begin{description}
\item[(a)] $X$ is equipped with a distance $d$ that makes $X$ complete;
\item[(b)] the application $\mathscr{Q}$ is well-defined: the solution $\rho\in X$ exists and is unique (for  a fixed $r$);
\item[(c)] the application $\mathscr{Q}$ is a contraction.
\end{description}
Once we have fullfilled these conditions, we can prove existence of a solution using a fixed point argument. 

Note that, in the modelization of pedestrian traffic,  the space dimension $N$ has to be equal to two, but our results are in fact true for all $N\in \naturali$. For instance, they can be adapted in dimension $N=3$ to modelize the behavior of fishes or birds. Consequently, we keep here a general $N$, even if we essentially think to the case $N=2$.

This note is organized as follows: in Section \ref{sec:model}, we describe some nonlocal models and their properties. In Section \ref{sec:k} we study one of these models through Kru\v zkov theory and in Section \ref{sec:ot} we study the same model through optimal transport theory.

%%%%%%%%%%%%%%%%%%%%%%%%%%%%%%%%%%%%%%%%%%%%%%%%%%%%%
\section{Pedestrian Traffic Modelization}\label{sec:model}
%%%%%%%%%%%%%%%%%%%%%%%%%%%%%%%%%%%%%%%%%%%%%%%%%%%%%
\subsection{One-Population model}
%%%%%%%%%%%%%%%%%%%%%%%%%%%%%%%%%%%%%%%%%%%%%%%%%%%%%
\subsubsection{Pedestrian in panic}\label{sec:panic}
The first model we present corresponds to pedestrians in panic and was studied in \cite{ColomboHertyMercier, ColomboLecureux}, in collaboration with R. M. Colombo and M. Herty. A panic phenomenon appears under special circumstances in crowded events. In these cases, the people are no longer rational and try, no matter how, to reach their target. 
Let us denote $\rho(t,x)$ the density of pedestrians at time $t$ and position $x\in \reali^N$. We consider the Cauchy problem:
\begin{equation}\label{eq:panic}
\pt_t \rho+\Div\left( \rho \, v(\eta* \rho(t,x))\,  \vec \nu(x)\right)=0\,;\qquad \rho(0,\cdot)=\rho_0\,.
\end{equation}
Here, $v$ is a real function describing the speed of the pedetrians. This function does not depend on the local density $\rho(t,x)$ but on the averaged density $\rho(t)*\eta(x)=\int_{\reali^N} \rho(t, x-y)\,\eta(y)\,\d{y}$. The vector field, $\vec \nu(x)$ describes the direction that the pedestrian located in $x$ will follow, independently from the distribution of the pedestrians' density. Note that we are working here on the all of $\reali^N$ and not on a subset of $\reali^N$; thus, we are \emph{not} working  on a  restriction to a room, for example. However, we can still introduce the presence of walls and obstacles in the choice of the vector field $\vec \nu$. 
Let us denote $\Omega \subset \reali^N$  the space where the pedestrians are authorized to walk, e.g. a room.
If we choose $\vec \nu(x)$ in a nice way (for example we can require that on the walls, i.e. for all  $x\in \pt\Omega$, $\vec \nu(x)$ coincides with the entering normal to $\Omega$), then we can  conclude to the invariance of the room. More precisely, if the initial density has support on some closed set $\Omega \subset \reali^N$, then the solution will have support contained in $\Omega$ for all time. This remark allows us to avoid considering  any boundaries and to have solutions on all $\reali^N$.

Using the Kru\v zkov theory on classical scalar conservation laws, we are able to prove:
\begin{theorem}[see \cite{ColomboLecureux}]\label{thm:panicK}
Let $\rho_0\in (\L1\cap\L\infty\cap\BV)(\reali^N, \rpic)$. Assume $v\in(\C2\cap \W2\infty)(\reali, \reali)$, $\vec \nu\in (\C2\cap\W21)(\reali^N, \reali^N)$, $\eta\in (\C2\cap\W2\infty)(\reali^N, \reali)$. Then there exists a unique weak entropy solution $\rho=S_t\rho_0\in \C0(\rpic, \L1(\reali^N, \rpic))$ to (\ref{eq:panic}) with initial condition $\rho_0$. 
Furthermore we have the estimate
\begin{equation}
\norma{\rho(t)}_{\L\infty}\leq \norma{\rho_0}_{\L\infty}e^{Ct}\,,
\end{equation}
where the constant $C$ depends on $v$, $\vec \nu$ and $\eta$.
\end{theorem}
For the definition of weak entropy solutions see Section \ref{sec:k}; the proof is defered to Section \ref{sec:proofK}.
Note that in Theorem \ref{thm:panicK},  the hypotheses are very strong. Let us denote $\mathcal{P}(\reali^N)$ the set of probability measures on $\reali^N$ and $\mathcal{M}^+(\reali^N)$ the set of positive measures on $\reali^N$. In collaboration with G. Crippa, using now some tools from optimal transport theory,  we obtained  the better result:
\begin{theorem}[see \cite{CrippaMercier}]\label{thm:panicOT}
Let $\rho_0\in \mathcal{M}^+(\reali^N)$. Assume $v\in (\L\infty\cap\lip)(\reali, \reali)$, $\vec \nu\in (\L\infty\cap\lip)(\reali^N, \reali^N)$, $\eta\in (\L\infty\cap\lip)(\reali^N, \rpic)$. Then there exists a unique weak measure solution $\rho\in \L\infty(\rpic, \mathcal{M}^+(\reali^N))$ to (\ref{eq:panic}) with initial condition $\rho_0$. 

If furthermore $\rho_0\in \L1(\reali^N, \rpic)$ then for all $t\geq 0$, the solution  $\rho$ satisfies also $\rho(t)\in \L1(\reali^N, \rpic)$.
\end{theorem}
For the definition of weak measure solution see Section \ref{sec:ot}; the proof is defered to Section \ref{sec:proofot}.

Note that for the model (\ref{eq:panic}), there is a priori no uniform $\L\infty$ bound on the density. Indeed, heuristically, considering the case in which the density is maximal, equal to 1, on the trajectory of a pedestrian located in $x$. If the averaged density around $x$ is strictly less than 1 (because, for example there is no one behind this pedestrian), then the speed $v(\rho*\eta)$ will be strictly positive, which means the pedestrian in $x$ will try to go forward, even though there is a queue in front of him. Consequently, we expect the density to become larger than one.

This behavior is not really unexpected in the case of panic. In fact, in some events the density attained up to 10 persons per square meter, which is obviously too much and a cause of deaths (see \cite{HelbingJohanssonZein}). Consequently, it is quite satisfactory to recover this behavior. One of our goal in this context is then to introduce a cost functional allowing to characterize the cases in which  the density is too high, and to find extrema of this functional. Let us introduce 
\[
J_T(\rho_0)=\int_0^T\int_{\Omega} f(S_t \rho_0)\d{x}\,,
\]
where $\Omega\subset\reali^N$ is the room, $\rho_0$ is the initial condition and $S_t\rho_0$ is the semi-group generated by Theorem \ref{thm:panicK}. We choose the function $f\in \C1(\reali, \rpic)$ so  that  it is equal to zero for any density $\rho$ less than a fixed threshold $\rho_c$ and  so that it is stictly increasing on $[\rho_c, +\infty[$. 
Consequently, the functional $J_T$ above allows to characterize the solutions with too high density and in particular it vanishes if the set $\{(t,x)\in [0,T]\times \reali^N \;:\; S_t\rho(x)\geq \rho_c\}$ has measure zero. We are then  interested in finding the minima of this cost functional. 

Using the Kru\v zkov theory we prove the following differentiability result:
\begin{theorem}[see \cite{ColomboHertyMercier, ColomboLecureux}]\label{thm:Gdiff}
Let $\rho_0\in (\W2\infty\cap \W21)(\reali^N, \rpic)$, $r_0\in (\W11\cap\L\infty)(\reali^N, \reali)$ and denote $\rho=S_t \rho_0$. Assume $v\in (\C4\cap\W2\infty)(\reali,\reali )$, $\vec \nu\in (\C3\cap\W21)(\reali^N, \reali^N)$, $\eta\in (\C3\cap \W2\infty)(\reali^N, \rpic)$.  Then there exists a unique weak entropy solution $r=\Sigma_t^\rho r_0$ to the Cauchy problem
\begin{equation}\label{eq:linearised}
\pt_t r+\Div (r \, v(\rho*\eta) \,\vec \nu(x))=-\Div ( \rho\, v'(\rho*\eta)\,\vec \nu(x))\,,\qquad r(0)=r_0\,.
\end{equation}

Furthermore, the semi-group $S_t$ obtained in Theorem \ref{thm:panicK} is Gâteaux-differentiable, that is to say, for all $\rho_0\in (\W21\cap\W2\infty)(\reali^N, \rpic)$, $ r_0\in(\W11\cap\L\infty)(\reali^N, \reali) $,
\begin{align*}
\lim_{h\to 0}\norma{\frac{S_t(\rho_0 +h r_0) -S_t\rho_0}{h}-\Sigma_t^\rho r_0}_{\L1}=0\,.
\end{align*}
\end{theorem} 
The proof is defered to Section \ref{sec:proofGdiff}. 

It is not possible to obtain the same result by the use of optimal transport theory. Indeed, it seems already not possible to find a good definition of Gâteaux differentiability  on the set of probability measures equipped with the Wasserstein distance of order 1.

\subsubsection{Orderly crowd}
In opposite to the previous model, for a model of orderly crowd it is required to have a uniform $\L\infty$ bound on the density. In collaboration with R. M. Colombo and M. Garavello \cite{ColomboGaravelloMercier}, we studied the equation 
\begin{equation}\label{eq:order}
\pt_t \rho+\Div\left[ \rho v(\rho) \big( \vec{\nu}(x) -  \frac{\nabla(\rho*\eta)}{ \sqrt{1+\norma{\nabla(\rho*\eta)}^2} } \big)\right] =0\,;
\end{equation}
with $\rho_0\in( \L1\cap\L\infty\cap\BV)(\reali^N;\reali)$.

In this model, the speed $v$ depends on the local density $\rho(t,x)$, which allows to prove some uniform bound in $\L\infty$. The prefered direction of the pedestrians  is still $\vec \nu (x)$, but they deviate from 
their optimal path trying to avoid entering regions with
higher densities. Indeed, $(\rho*\eta)$ is an average of the crowd density around $x$ and $-\nabla (\rho*\eta)$ is a vector going in the direction opposite to the area of maximal  averaged density. Due to the nonlinearity of the flow with respect to  the local $\rho$, it is no longer possible to use optimal transport theory. Hence, we use Kru\v zkov theory, to prove existence and uniqueness of solutions. We obtain the theorem:
\begin{theorem}[see \cite{ColomboGaravelloMercier}]\label{thm:order}
Assume that $v\in \C2([0,1], \rpic)$ satisfies $v(1)=0$, that $\vec \nu \in (\C2 \cap\W21\cap\W1\infty)(\reali^N, \reali^N)$,  and that $\eta \in (\C3\cap\W31\cap\W2\infty)(\reali^N, \reali)$. Then, for any $\rho_0 \in (\L1\cap\L\infty\cap\BV)(\reali^N , [0,1])$, there exists a unique weak entropy solution $\rho\in \C0(\rpic, \L1(\reali^N, [0,1]))$. 
\end{theorem}
The proof of this theorem relies on Kru\v zkov theory (see Section \ref{sec:k}). 
Note that for this model the density is uniformly bounded in $\L\infty$, contrarily to the panic model, in which the $\L\infty$ norm can grow exponentially in time.

Some difficulties now appear in proving that the pedestrians remain in the authorized area. Let us introduce the following invariance property:
\begin{description}
\item[(P)] Let $\Omega\subset \reali^N$ be region where the pedestrians are allowed to walk. The model (\ref{eq:order}) is invariant with respect to $\Omega$ if
\begin{equation}
\supp( \rho_0) \subset \Omega \qquad \Rightarrow \qquad \supp( \rho(t))\subset \reali^N \quad \textrm{ for all }t\geq 0\,.
\end{equation}
\end{description}
To obtain that \textbf{(P)} is satisfied, we have to require the prefered direction to be strongly entering the room (see \cite[Proposition 3.1 \& Appendix A]{ColomboGaravelloMercier}).

Some interesting phenomena show up through numerical computations. First, when considering a crowd walking along a corridor, we observe the formation of lanes (see Figure \ref{fig:lanes}).
\begin{figure}[htpb]
  \centering
  \includegraphics[width=0.32\textwidth, trim=75 120 20
  120]{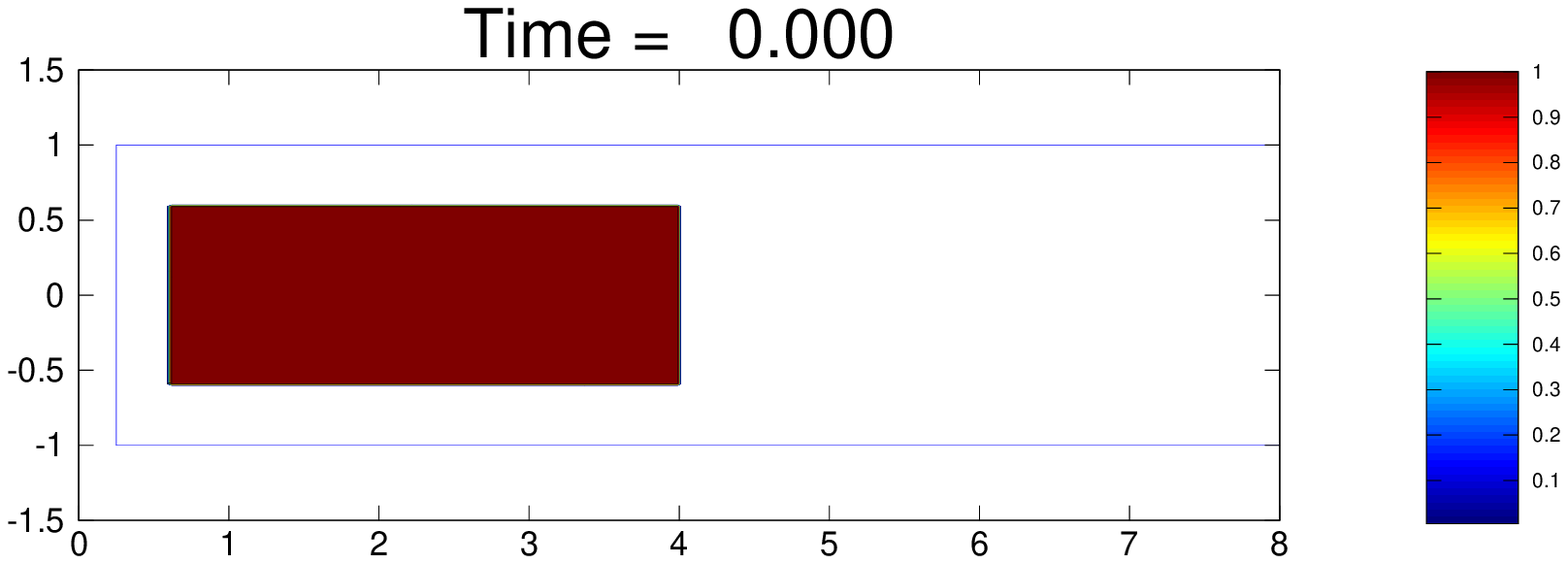}\hfil%
  \includegraphics[width=0.32\textwidth, trim=75 120 20
  120]{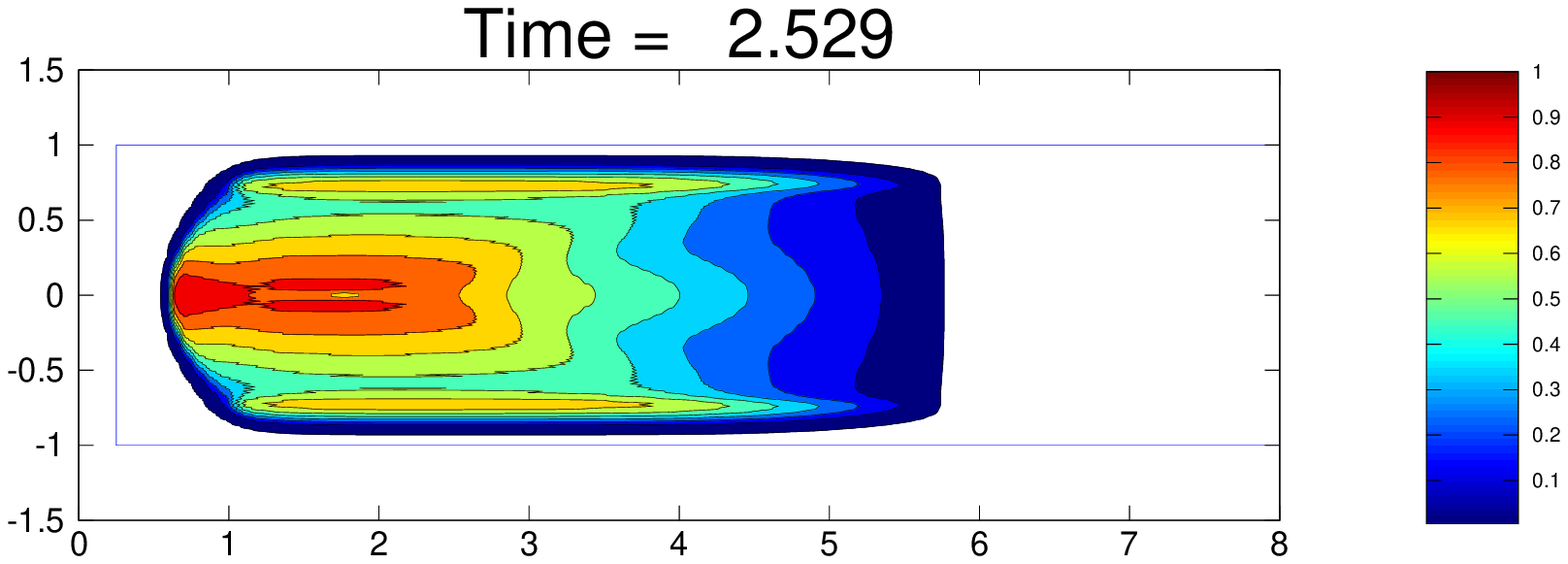}\hfil%
  \includegraphics[width=0.32\textwidth,
  trim=75 120 20 120]{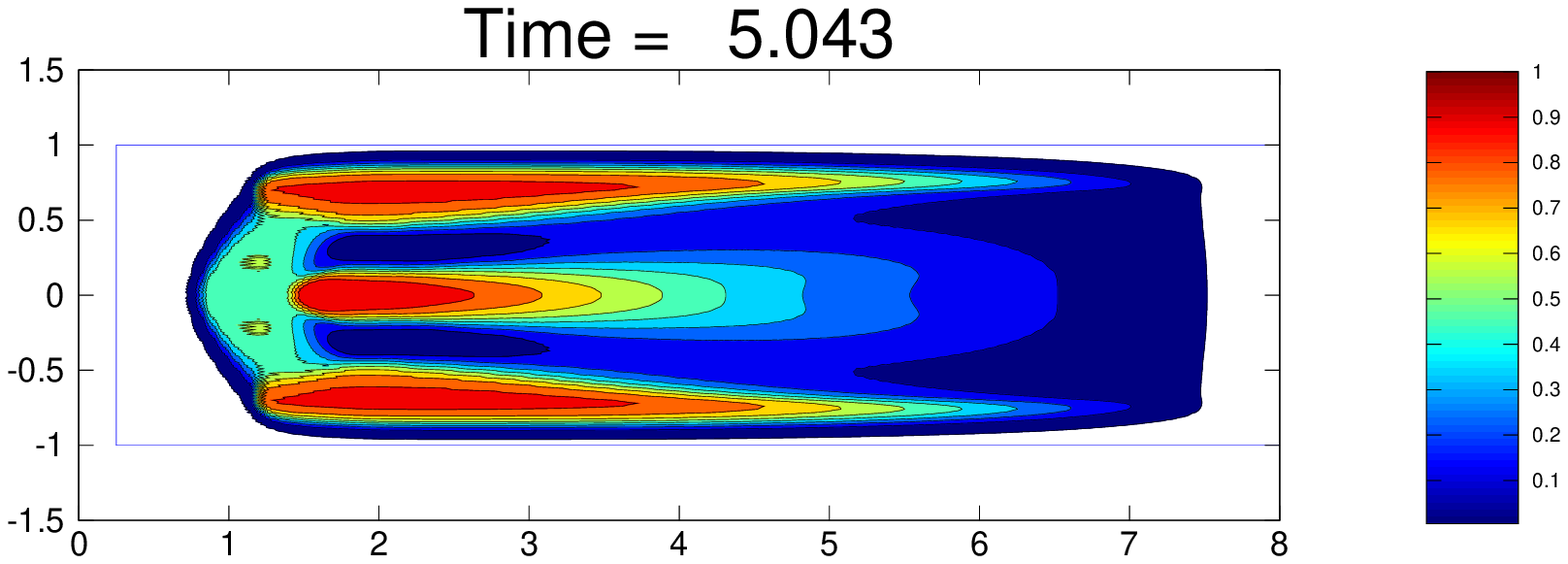}\\
  \includegraphics[width=0.32\textwidth, trim=75 120 20
  120]{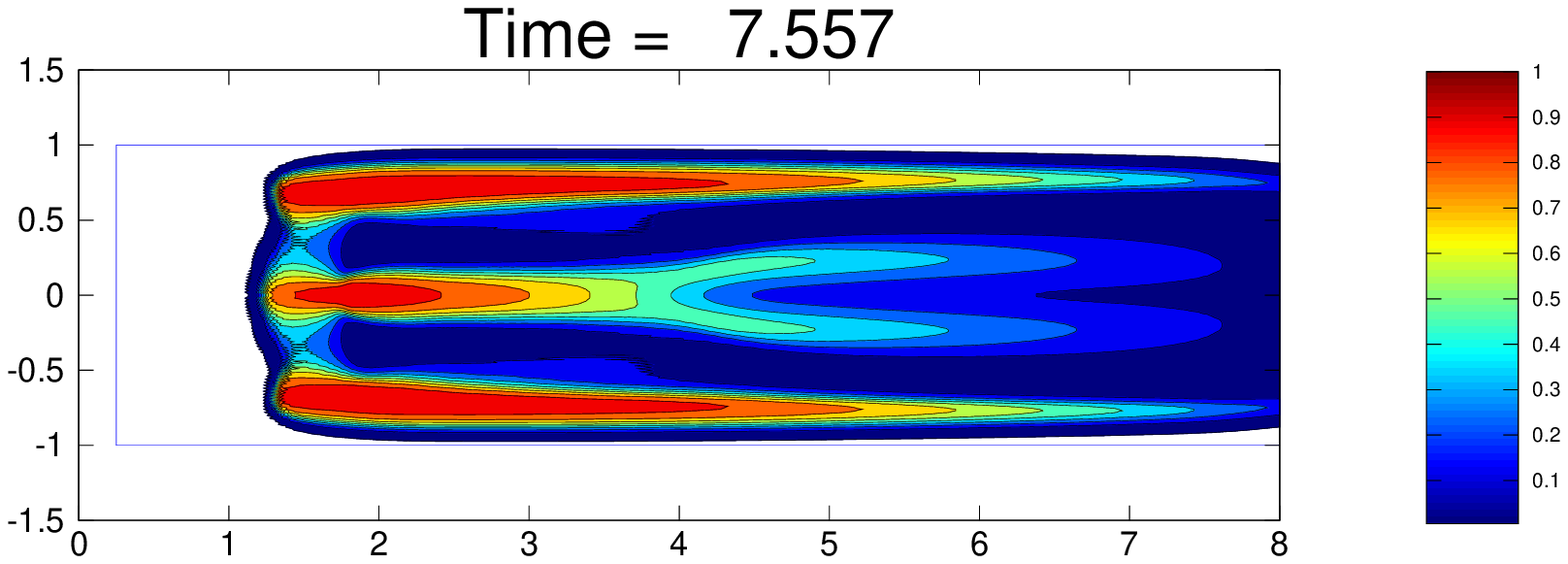}\hfil%
  \includegraphics[width=0.32\textwidth, trim=75 120 20
  120]{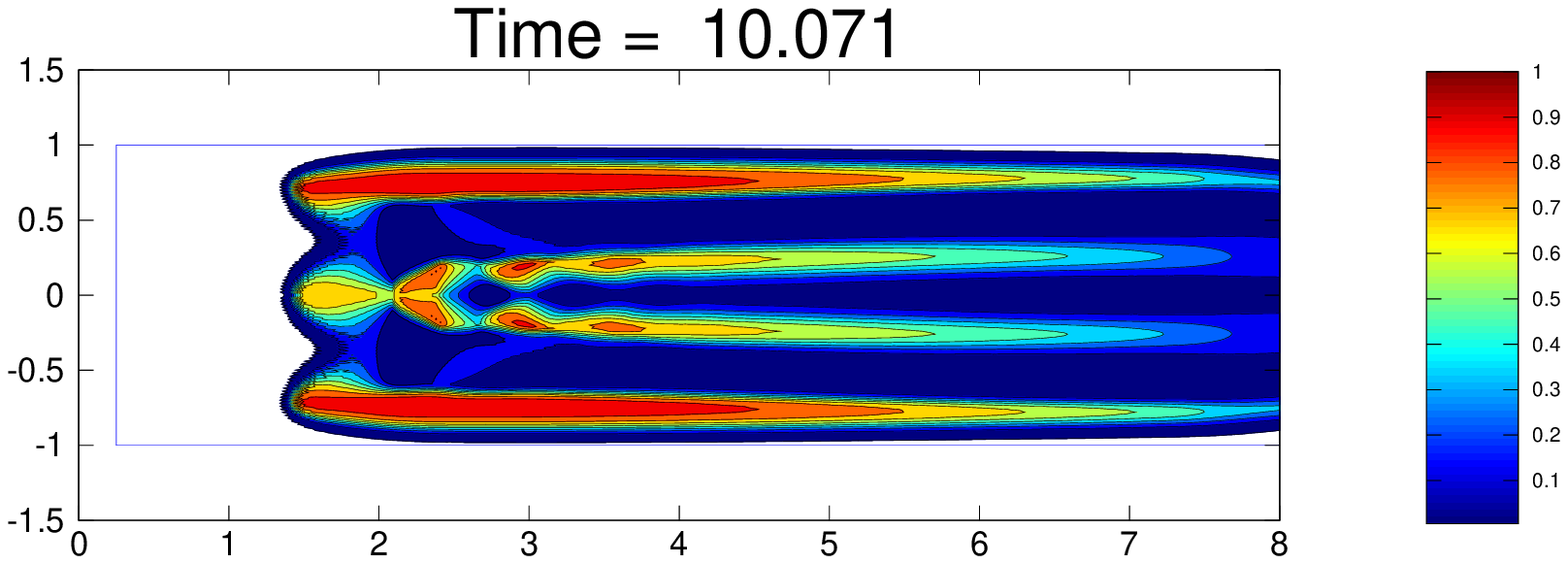}\hfil%
  \includegraphics[width=0.32\textwidth, trim=75 120 20
  120]{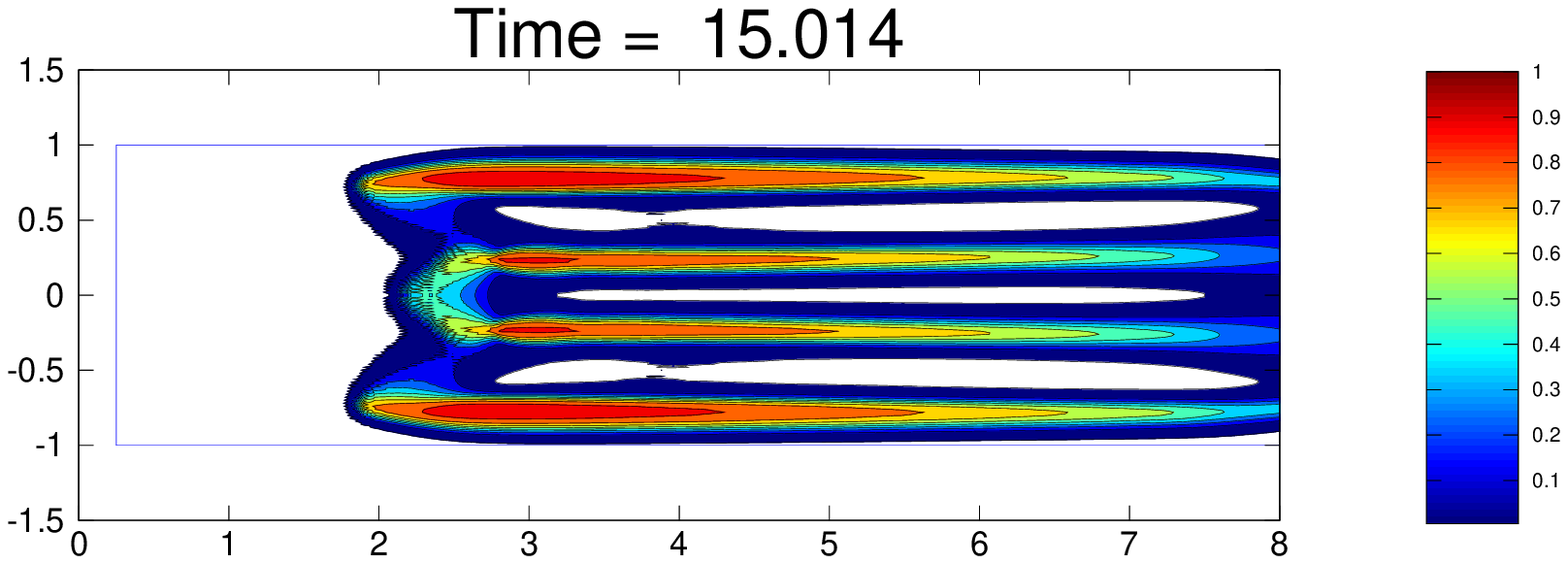}%
  \caption{Solution
    to~\eqref{eq:order}
    at times $t=0$, $2.529$, $5.043$, $7.557$, $10.071,\,
    15.014$. First 3 lanes are formed, then the middle lane bifurcates
    forming the fourth lane. Picture from \cite{ColomboGaravelloMercier}.}
  \label{fig:lanes}
\end{figure}

Furthermore, this phenomenon seems very stable with respect to initial conditions and geometry. Indeed considering the room and initial distribution as in Figure \ref{fig:Initial}, adding some various obstacles, we still have lanes, at least in large space (see Figures \ref{fig:Evacuation1}). 
\begin{figure}[h!]
\begin{center}
    \includegraphics[width=0.3\textwidth, trim=60 120 50
    120]{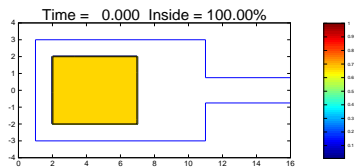}
    \caption{Initial datum and room geometry. Picture from \cite{ColomboGaravelloMercier}.}
    \label{fig:Initial}
    \end{center}
\end{figure}

\begin{figure}[h!]
  \centering
  \includegraphics[width=0.3\textwidth, trim=75 110 10 120]{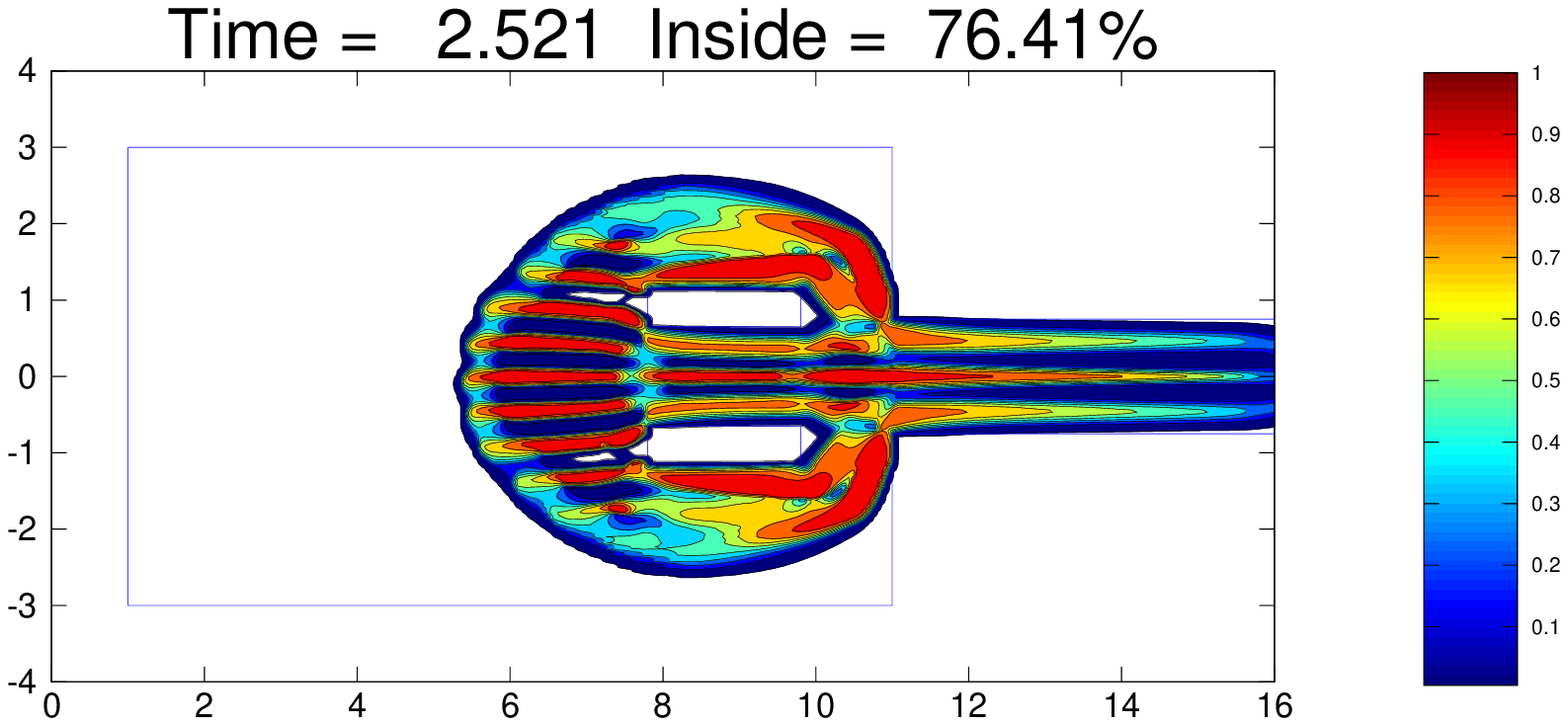}%
  \includegraphics[width=0.3\textwidth, trim=75 110 10 120]{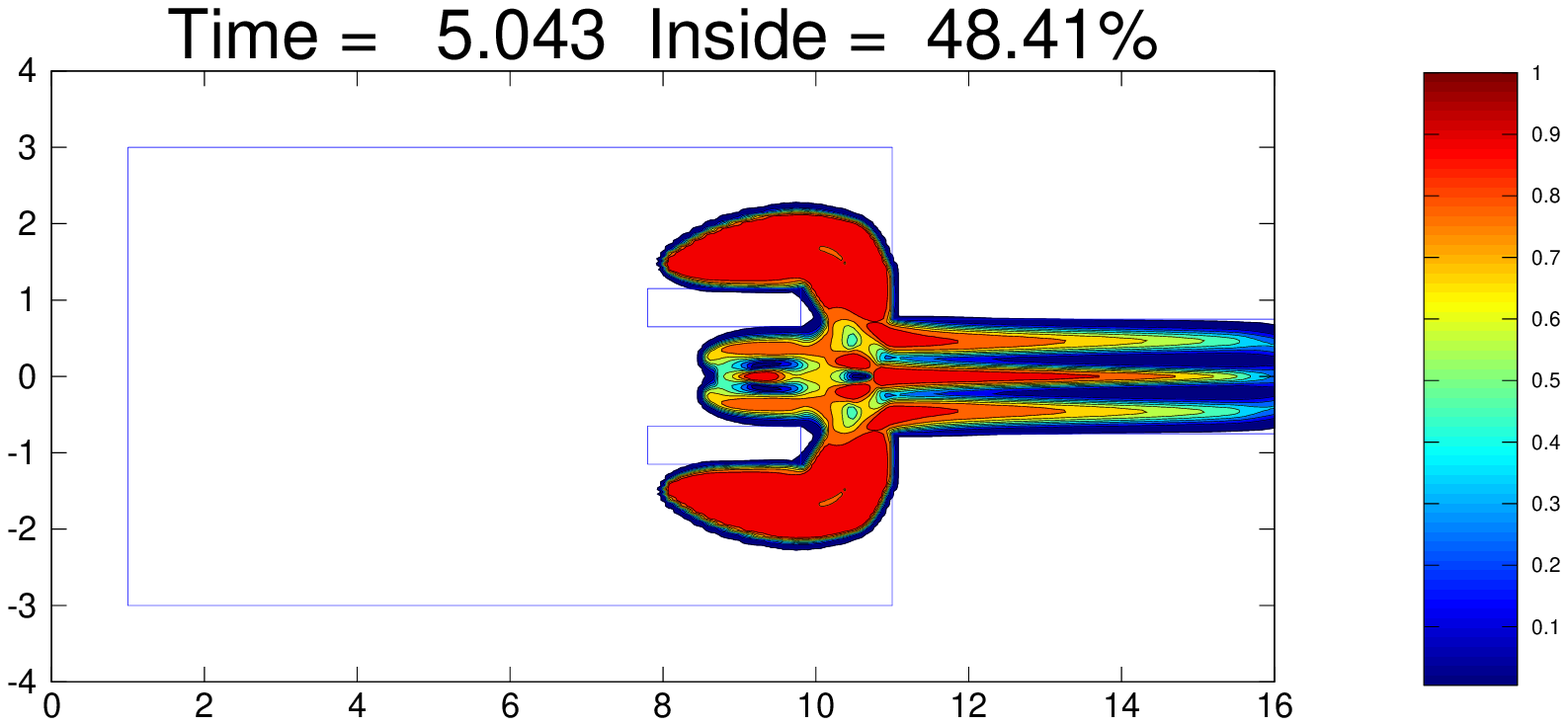}%
  \includegraphics[width=0.3\textwidth, trim=75 110 10 120]{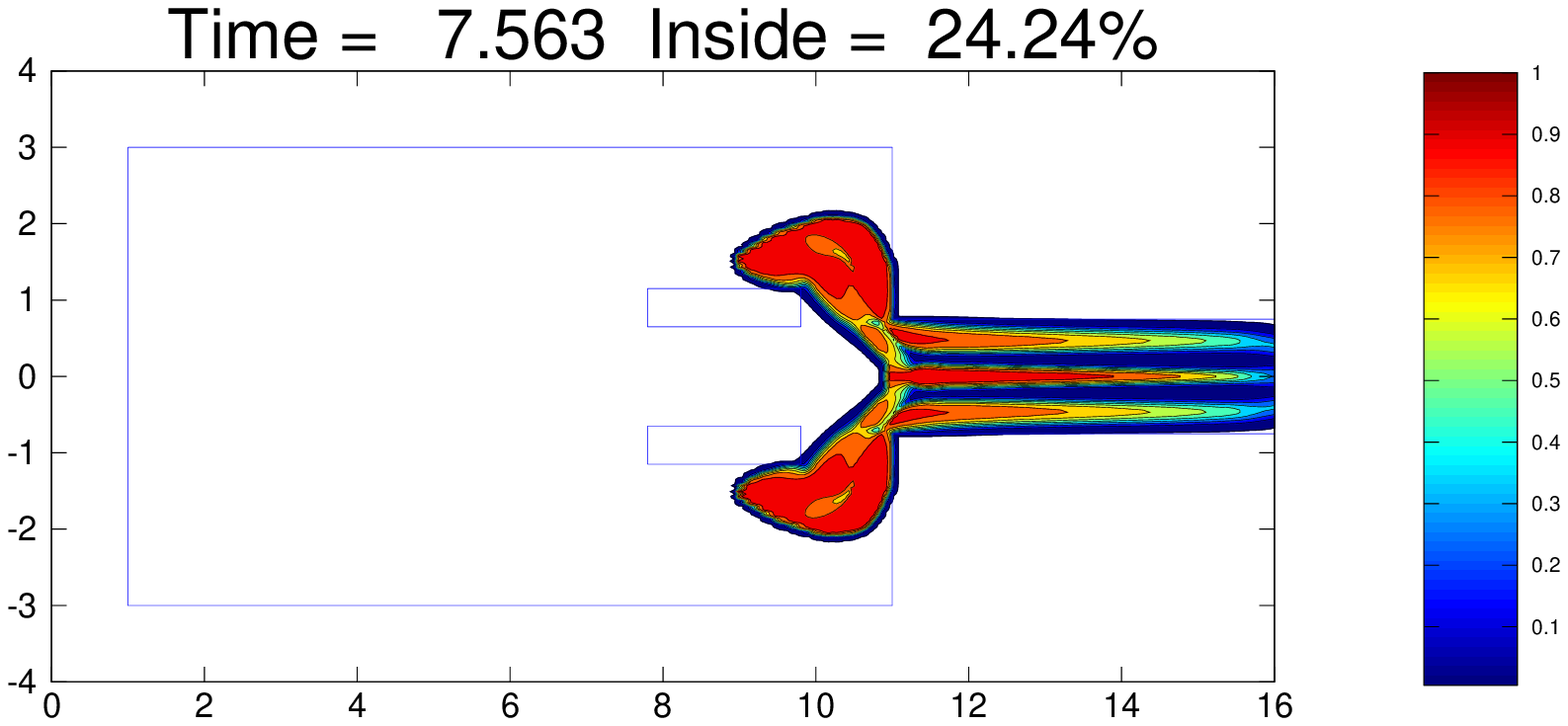}\\[0.7cm]
  \includegraphics[width=0.3\textwidth, trim=75 120 10 110]{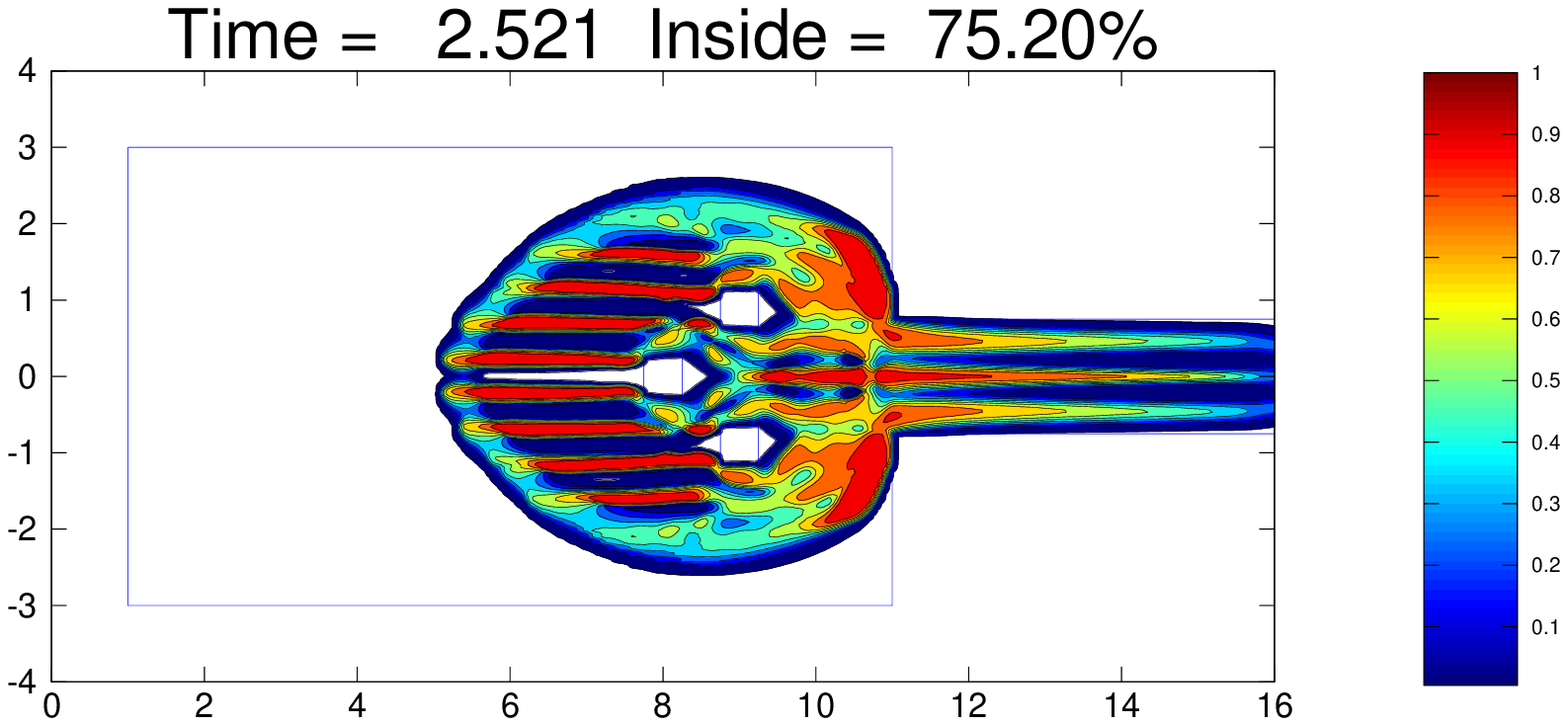}
  \includegraphics[width=0.3\textwidth, trim=75 120 10 110]{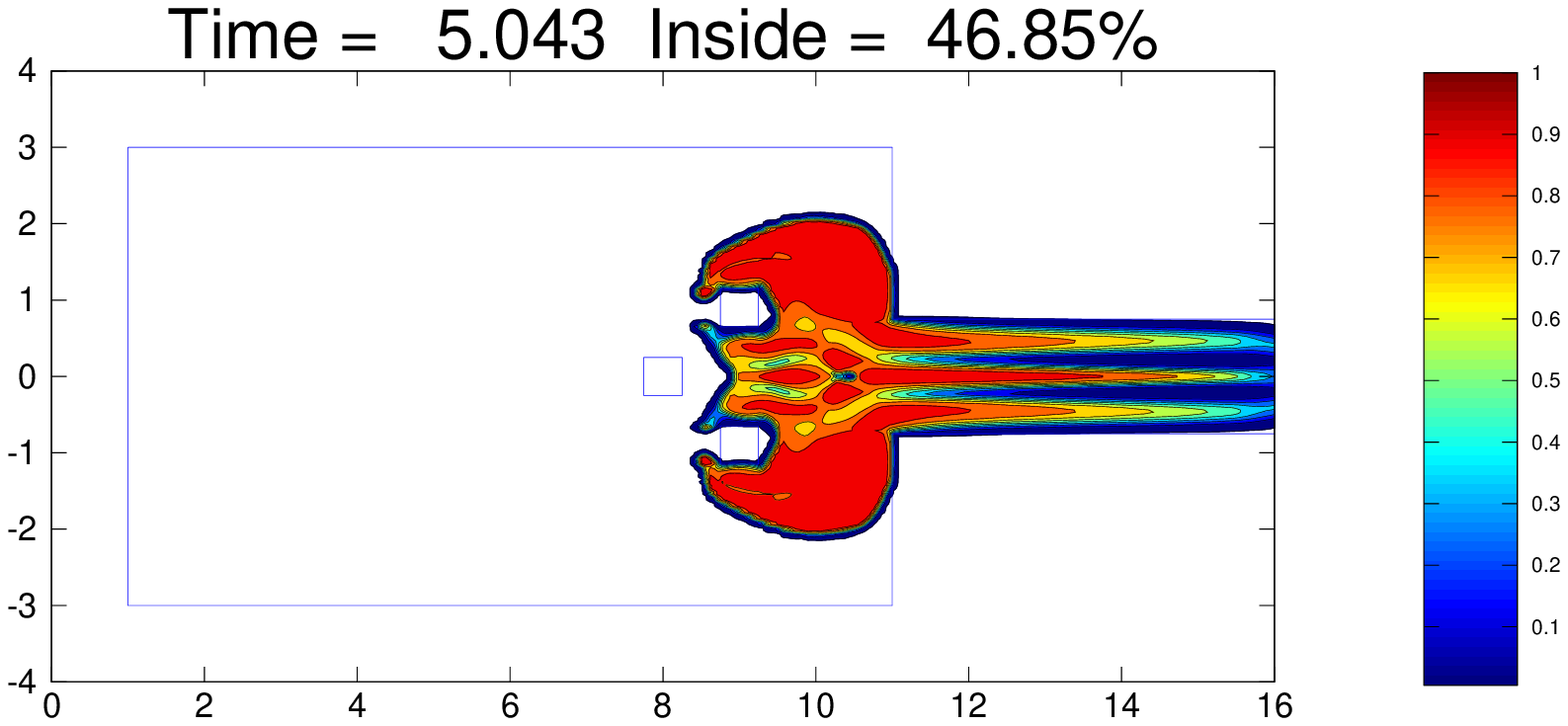}%
  \includegraphics[width=0.3\textwidth, trim=75 120 10 110]{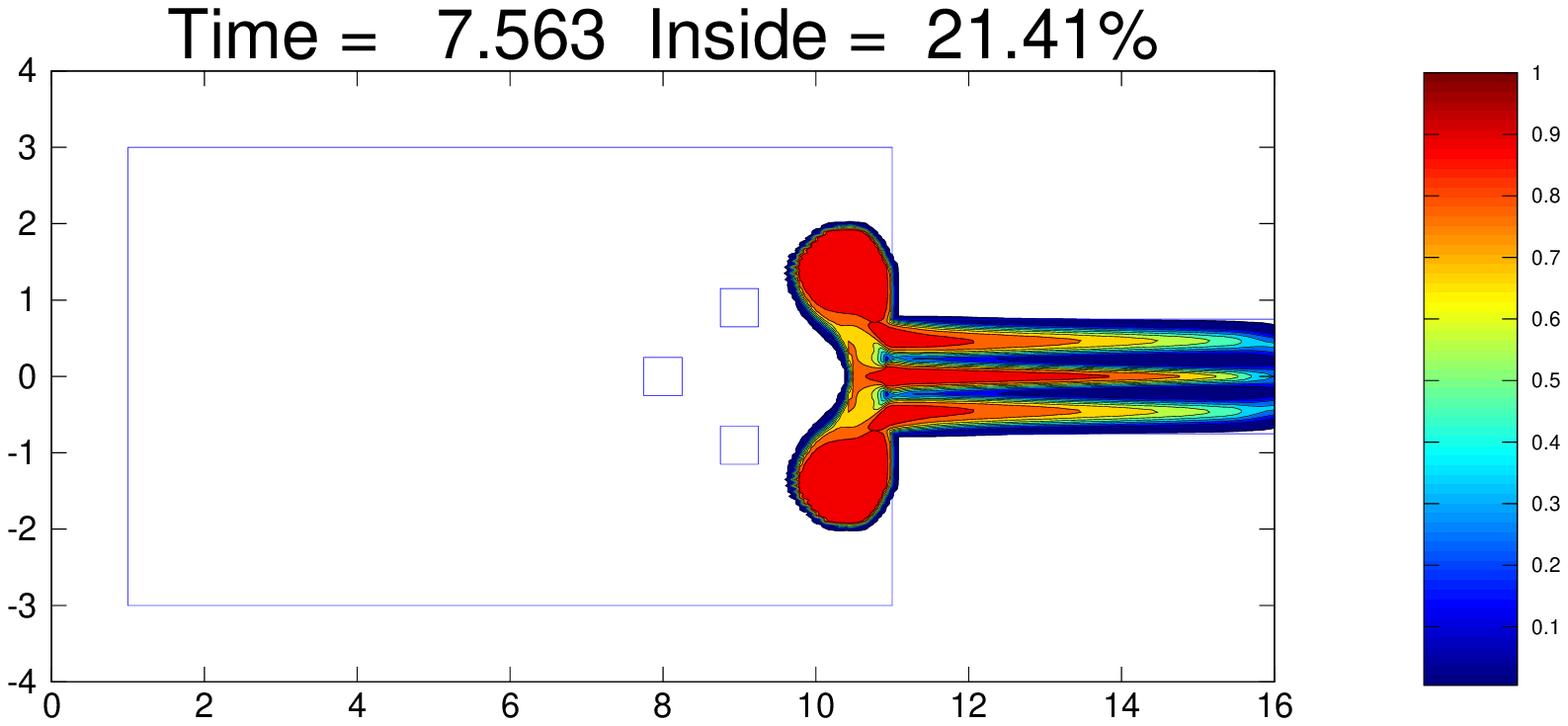}%
  \caption{Solution
 to~(\ref{eq:order})
    with different geometries, computed at time $t = 2.521$, $5.043$
    and $7.563$. Picture from \cite{ColomboGaravelloMercier}.}
  \label{fig:Evacuation1}
\end{figure}

%Through numerical integrations (see firgure \ref{fig:Evacuation1}) we provide
%examples of solutions to~\eqref{eq:General}--\eqref{eq:IGood}. They
%show the interesting phenomenon of \emph{pattern formation}. In the
%case of a crowd walking along a corridor, coherently with the
%experimental observation described in the literature, see for
%instance~\cite{HelbingEtAl2002, HelbingJohansson2007,
%  HoogendoornBovy2003, PiccoliTosin2009}, the solution
%to~\eqref{eq:General}--\eqref{eq:IGood} self-organizes into lanes. The
%width of these lanes depends on the size of the support of the
%averaging kernel $\eta$. 
%This feature is stable with respect to strong
%variations in the initial datum and also in the geometry. Indeed, we
%have lane formation also in the case of the evacuation of a room, when
%the crowd density sharply increases in front of the
%door. 

A further remarkable property of the
model (\ref{eq:order}) is that it captures the
following well-known, although sometimes counter intuitive 
phenomenon (Braess paradox): the evacuation time through an exit can be reduced by
carefully positioning suitable \emph{``obstacles''} that direct the
outflow (see for instance~\cite{HelbingJohansson} and the references
therein). Indeed, looking at Figure \ref{fig:braess}, we observe that the time of exist with obstacles is slightly shorter than the one without any obstacle.
\begin{figure}[h!]
  \centering
  \includegraphics[width=0.3\textwidth, trim=75 110 10 120]{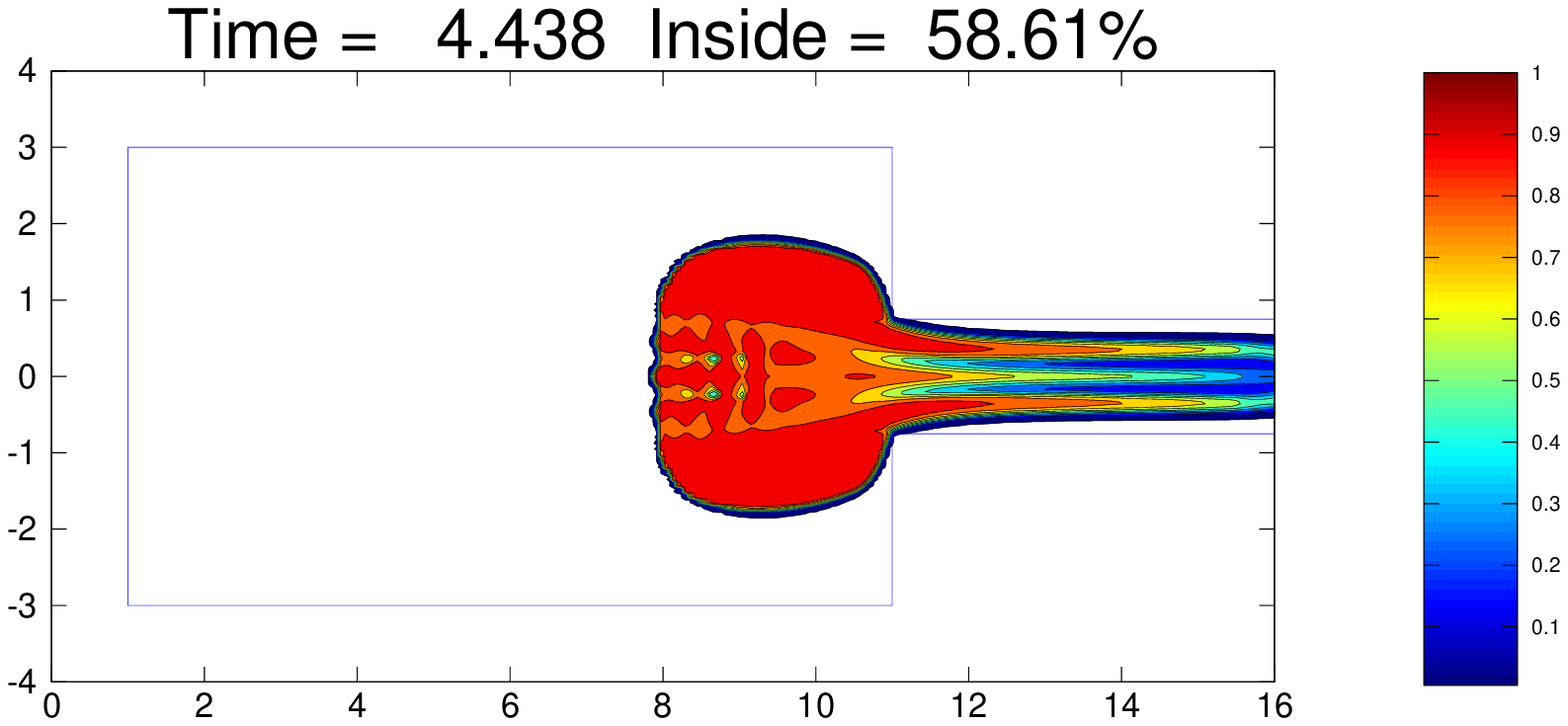}%
  \includegraphics[width=0.3\textwidth, trim=75 110 10 120]{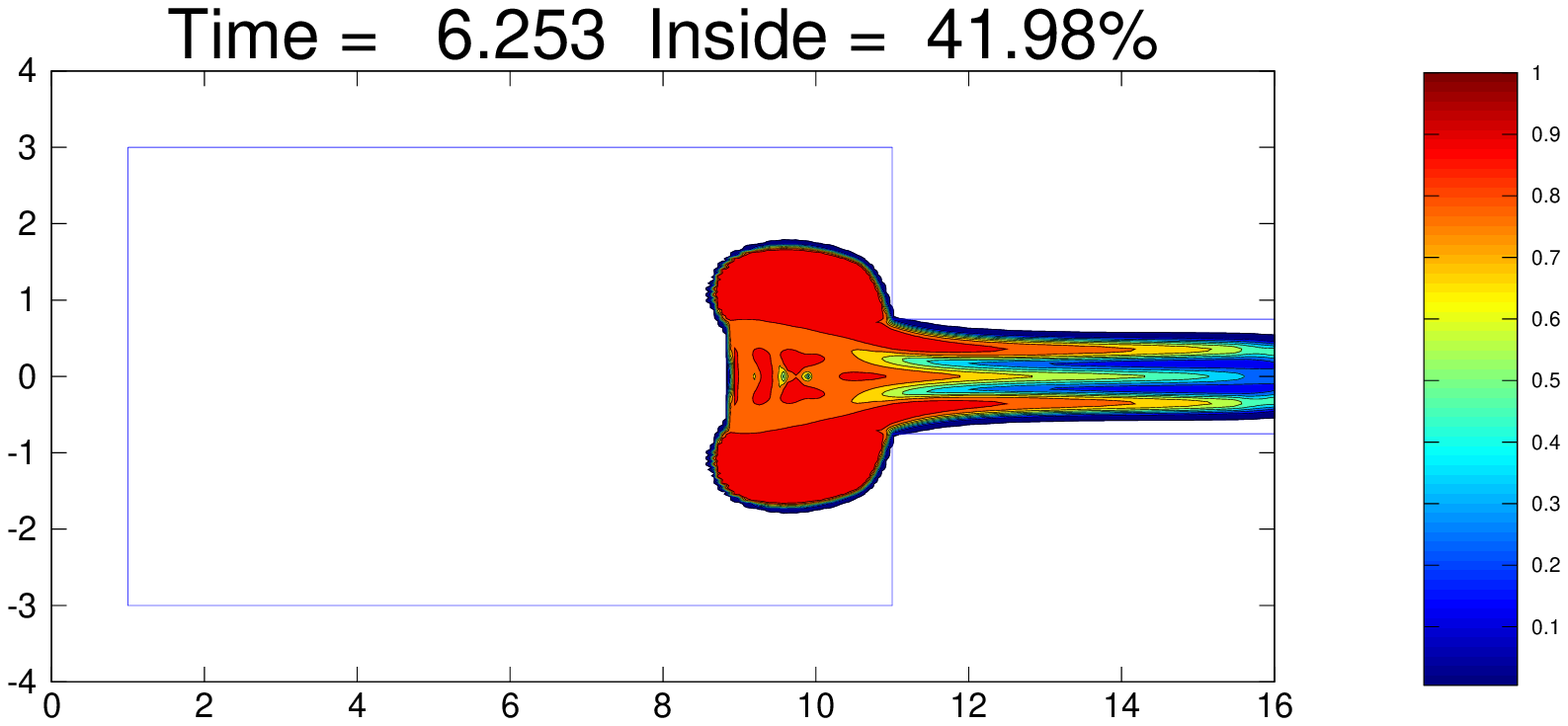}%
  \includegraphics[width=0.3\textwidth, trim=75 110 10 120]{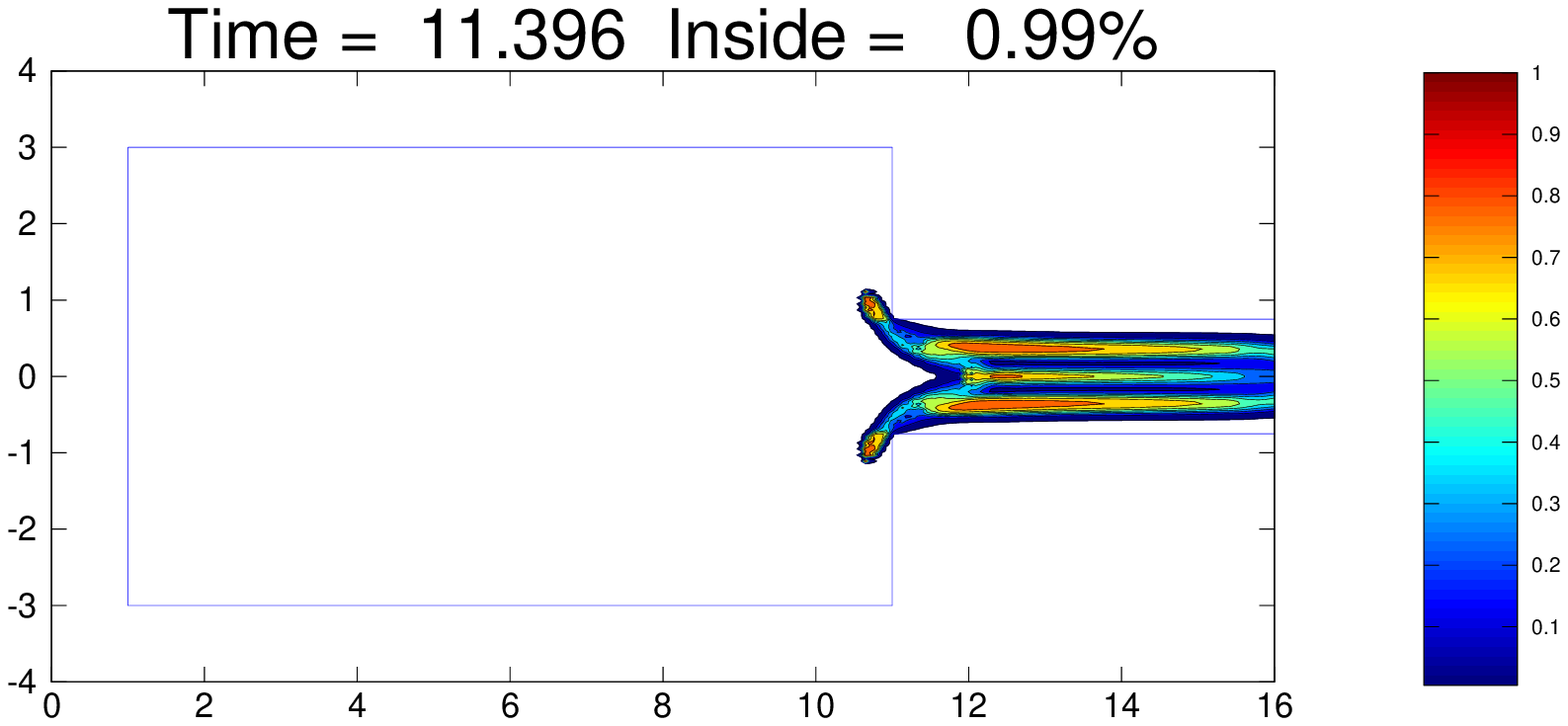}\\[0.7cm]
  \includegraphics[width=0.3\textwidth, trim=75 120 10 110]{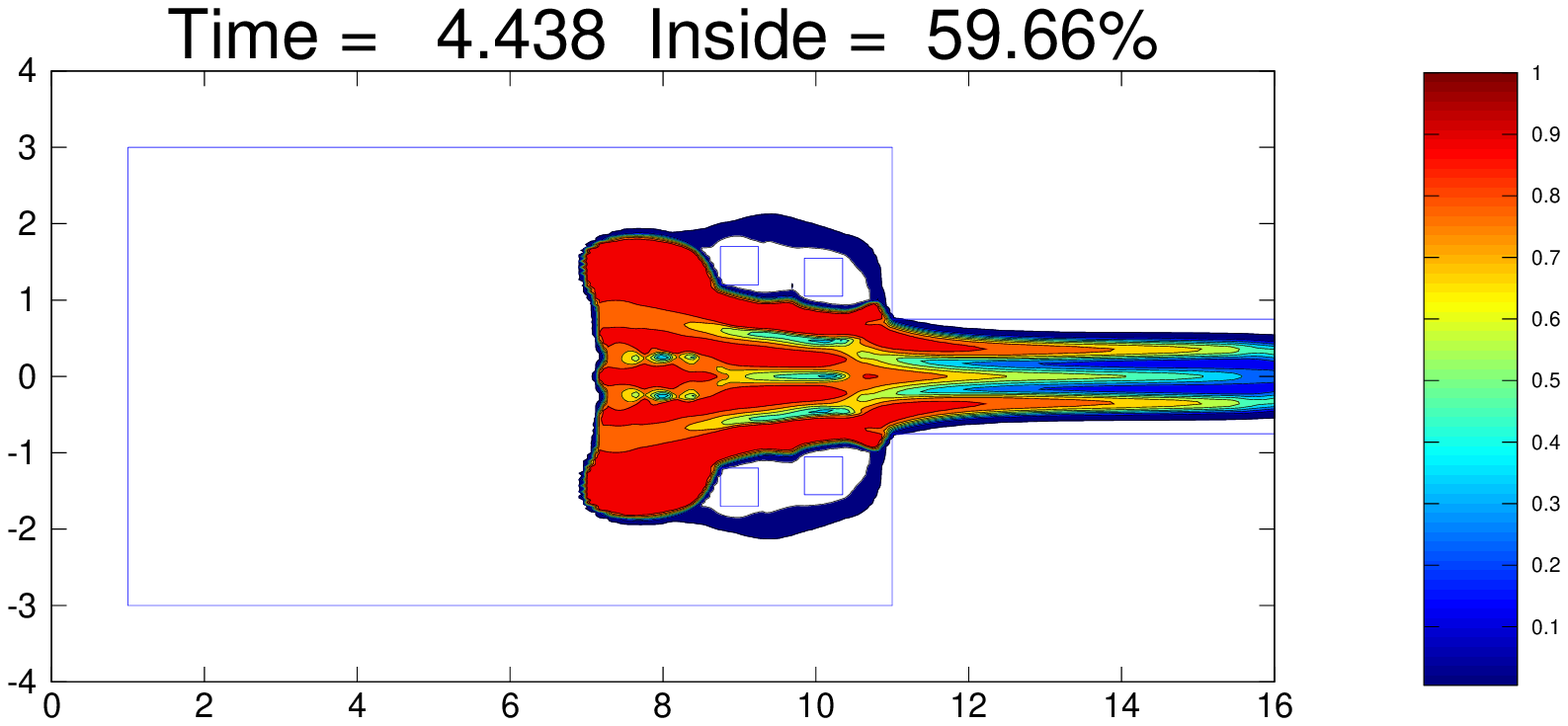}%
  \includegraphics[width=0.3\textwidth, trim=75 120 10 110]{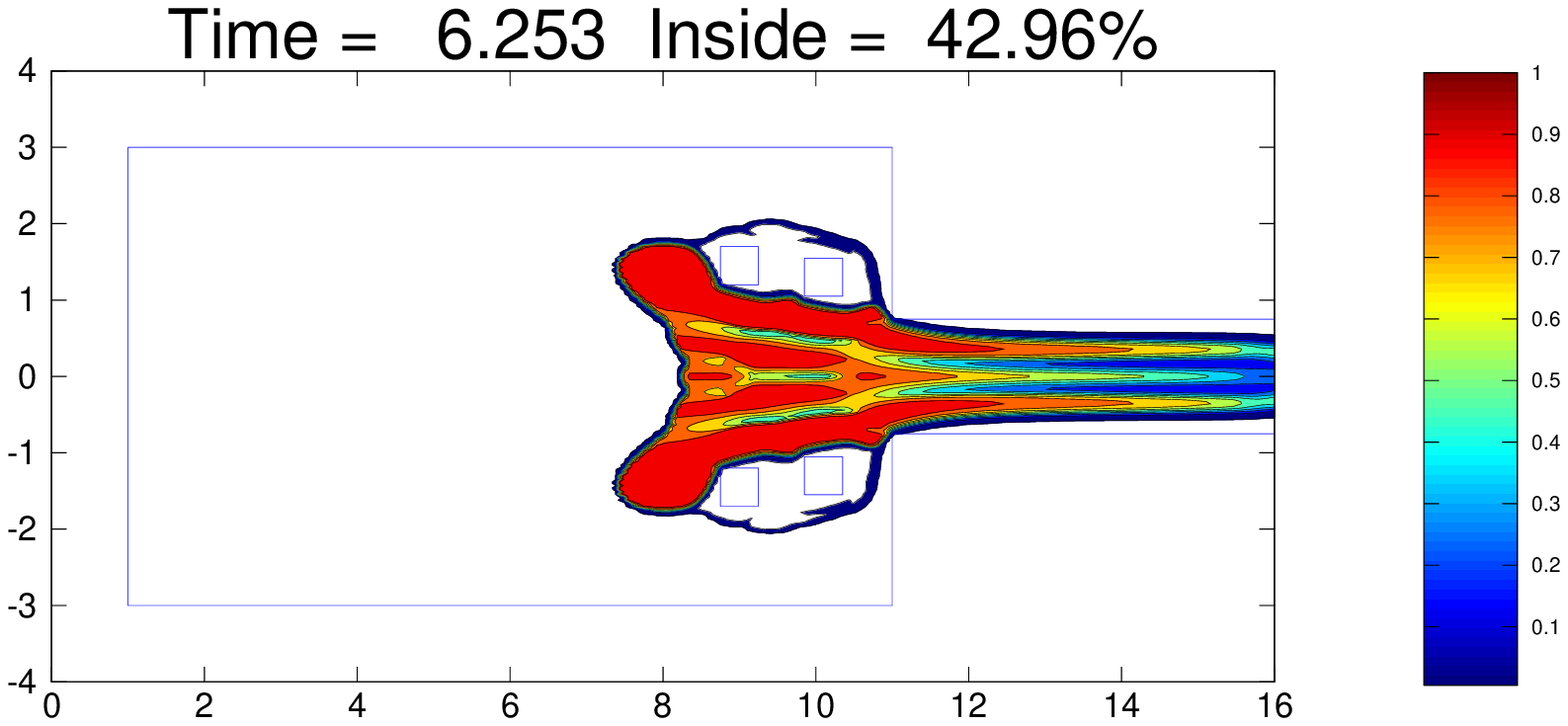}%
  \includegraphics[width=0.3\textwidth, trim=75 120 10 110]{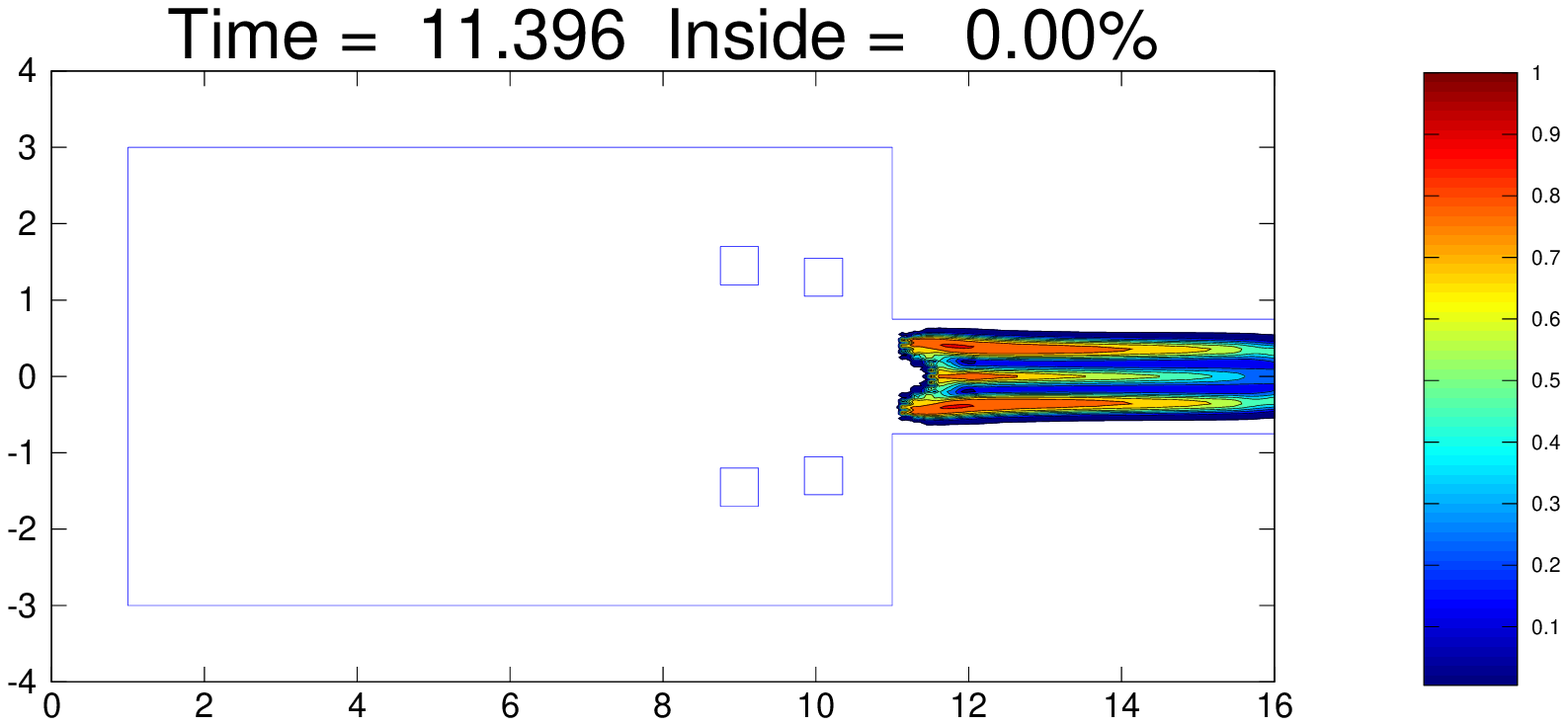}
    \caption{Solution
    to~(\ref{eq:order})
    with $\epsilon=0.2$, at times $t=4.438$, $6.253$, $11.396$. On the
    first line, no obstacle is present. On the second line, $4$
    columns direct the crowd flow. The evacuation time in
    the latter case is \emph{shorter} than in the former one. Picture from \cite{ColomboGaravelloMercier}.}\label{fig:braess}
\end{figure}

\subsection{Several populations}
A natural wish now is to extend the previous models to the case of several population with different objectives.

\subsubsection{Panic}
With several populations, we have to consider several densities and several equations. 
For two populations, extending the idea of the equation (\ref{eq:panic}) to the case of several populations, we obtain the system
\begin{equation}\label{eq:panic2}
\left\{
\begin{array}{l}
\pt_t \rho_1 +\Div\left(\rho_1 \,v(\rho_1*\eta_1+\rho_2*\eta_2)\,\vec \nu_1(x)\right)=0\,,\\
\pt_t \rho_2 +\Div\left(\rho_2 \, v(\rho_1*\eta_1+\rho_2*\eta_2)\,\vec \nu_2(x)\right)=0\,.
\end{array}\right.
\end{equation}
Here we consider that the speed $v$ depends on the  average of the total density $\rho_1+\rho_2$. Furthermore, the difference of goals of the populations 1 and 2 is reflected in the choice of different prefered directions $\vec \nu_1$ and $\vec \nu_2$.

We are able to prove theorems similar to the ones of section \ref{sec:panic}.  In collaboration with R. M. Colombo \cite{ColomboLecureux}, we obtained:
\begin{theorem}[see \cite{ColomboLecureux}]\label{thm:panicK2}
 Let $\rho_0=(\rho_{0,1}, \rho_{0,2})\in (\L1\cap\L\infty\cap\BV)(\reali^N, \rpic^2)$. Assume that, for $i\in \{1,2\}$, $v_i\in (\C2\cap\W2\infty)(\reali, \reali)$, $\vec \nu_i\in (\C2\cap\W21)(\reali^N, \reali^N)$, $\eta_i\in (\C2\cap\W2\infty)(\reali^N, \reali)$. Then there exists a unique weak entropy solution $\rho=S_t\rho_0\in \C0(\rpic, \L1(\reali^N, \rpic^2))$ to (\ref{eq:panic}) with initial condition $\rho_0$. 
Furthermore we have the estimate
\begin{equation}
\norma{\rho(t)}_{\L\infty}\leq \norma{\rho_0}_{\L\infty}e^{Ct}\,,
\end{equation}
where the constant $C$ depends on $v$, $\vec \nu$ and $\eta$
\end{theorem}
Similarly as for one population, it is possible with several populations to prove the Gâteaux-differentiability thanks to Kru\v zkov theory (see \cite[Theorem 2.2]{ColomboLecureux}).
%For the definition of weak entropy solution see Section \ref{sec:k}; the proof is defered to Section \ref{sec:proofK}.

As in Theorem \ref{thm:panicK2},  the hypotheses  of Theorem \ref{thm:panicK2} are very strong. In collaboration with G. Crippa \cite{CrippaMercier}, using  tools from optimal transport theory, we obtained the better result:
\begin{theorem}[see \cite{CrippaMercier}]\label{thm:panicOT2}
Let $\rho_0\in \mathcal{P}(\reali^N)^2$. Assume $v_i\in (\L\infty\cap\lip)(\reali, \reali)$, $\vec \nu_i\in (\L\infty\cap\lip)(\reali^N, \reali^N)$, $\eta_i\in (\L\infty\cap\lip)(\reali^N, \rpic)$. Then there exists a unique weak measure solution $\rho\in \L\infty(\rpic, \mathcal{P}(\reali^N)^2)$ to (\ref{eq:panic2}) with initial condition $\rho_0$. 

If furthermore $\rho_0\in \L1(\reali^N, \rpic)$ then for all $t\geq 0$, we have $\rho(t)\in \L1(\reali^N, \rpic)$.
\end{theorem}
The idea of the proof for this theorem is given in Section \ref{sec:ot}.

\textbf{Interaction continuum / individuals.}
Note that in the framework of Theorem \ref{thm:panicOT2}, we are dealing with measure solutions. This context allows us to describe a coupling between a group of density $\rho_1$ and an individual located in $p(t)$. Indeed, let us assume that  $\rho_1\in \L\infty(\rpic, \L1(\reali^N, \rpic))$ and 
 $\rho_2=\delta_{p(t)}$ is a Dirac measure.
Then we have the following ODE/PDE coupling:
\[
\left\{
\begin{array}{l}
 \pt_t \rho_1+\div \left(\rho_1 \, v\left(   \rho_1*\eta_1(x)+ \eta_2(x-p(t)) \right)\vec \nu_1(x) \right)=0\,,\\[5pt]
 \dot p(t)=v\left(\rho_1*\eta_1\left(p(t)\right)+ \eta_2(0)\right) \vec \nu_2(p(t))\,.\\
\end{array}
\right.
\]

\subsubsection{Orderly crowd}
We now extend (\ref{eq:order}) to the case of several populations. 
We consider here not only that  $\rho_1$ and $\rho_2$ have two different goals, but also that $\rho_1$ is repelled by $\rho_2$, and that $\rho_2$ is repelled by $\rho_1$. Let us denote $\epsilon_i$ the parameter of self-interaction and $\epsilon_o$ the parameter of interaction with the other population. We obtain:
\begin{equation}\label{eq:order2}
  \left\{
    \begin{array}{rcl}
   \displaystyle   \pt_t\rho_1
      +
      \Div \big(
        \rho_1 \, v_1(\rho_1)
        \big(\vec \nu_1(x)- \epsilon_i \frac{\nabla\rho_1*\eta_1}{\sqrt{1+\norma{\nabla \rho_1*\eta_1}^2}}- \epsilon_o \frac{\nabla\rho_2*\eta_2} {\sqrt{1+\norma{\nabla \rho_2*\eta_2}^2}}
        \big)
      \big)
      & = & 0\,,
      \\[0.7cm]
 \displaystyle     \pt_t\rho_2
      +
      \Div \big(
        \rho_2 \, v_2(\rho_2)
        \big(\vec \nu_2(x)-\epsilon_o \frac{\nabla\rho_1*\eta_1}{\sqrt{1+\norma{\nabla \rho_1*\eta_1}^2}}- \epsilon_i \frac{\nabla\rho_2*\eta_2} {\sqrt{1+\norma{\nabla \rho_2*\eta_2}^2}}\big)
      \big)
      & = & 0\,.
    \end{array}
  \right.
\end{equation}

\begin{remark}
Note that the interaction  between the equations is only in the nonlocal term.
\end{remark}

We obtain again existence and uniqueness of weak entropy solutions thanks to Kru\v zkov theory (see \cite[Theorem 3.2]{ColomboLecureux}). The obtained theorem is similar to Theorem \ref{thm:order} so we don't rewrite it.

A classical situation considered in the engineering literature is that of two groups of people 
moving in opposite directions and crossing each other. The numerical integration (see Figure \ref{fig:NI1}), shows formation of lanes that are not superimposing as described in the engineering literature (\cite{HelbingEtAlii2001, DaamenHoogendoorn2003} or~\cite{DegondEtAl} ). 
\begin{figure}[htpb]
  \centering%
  \includegraphics[width=0.25\textwidth, trim=120 40 121 30]{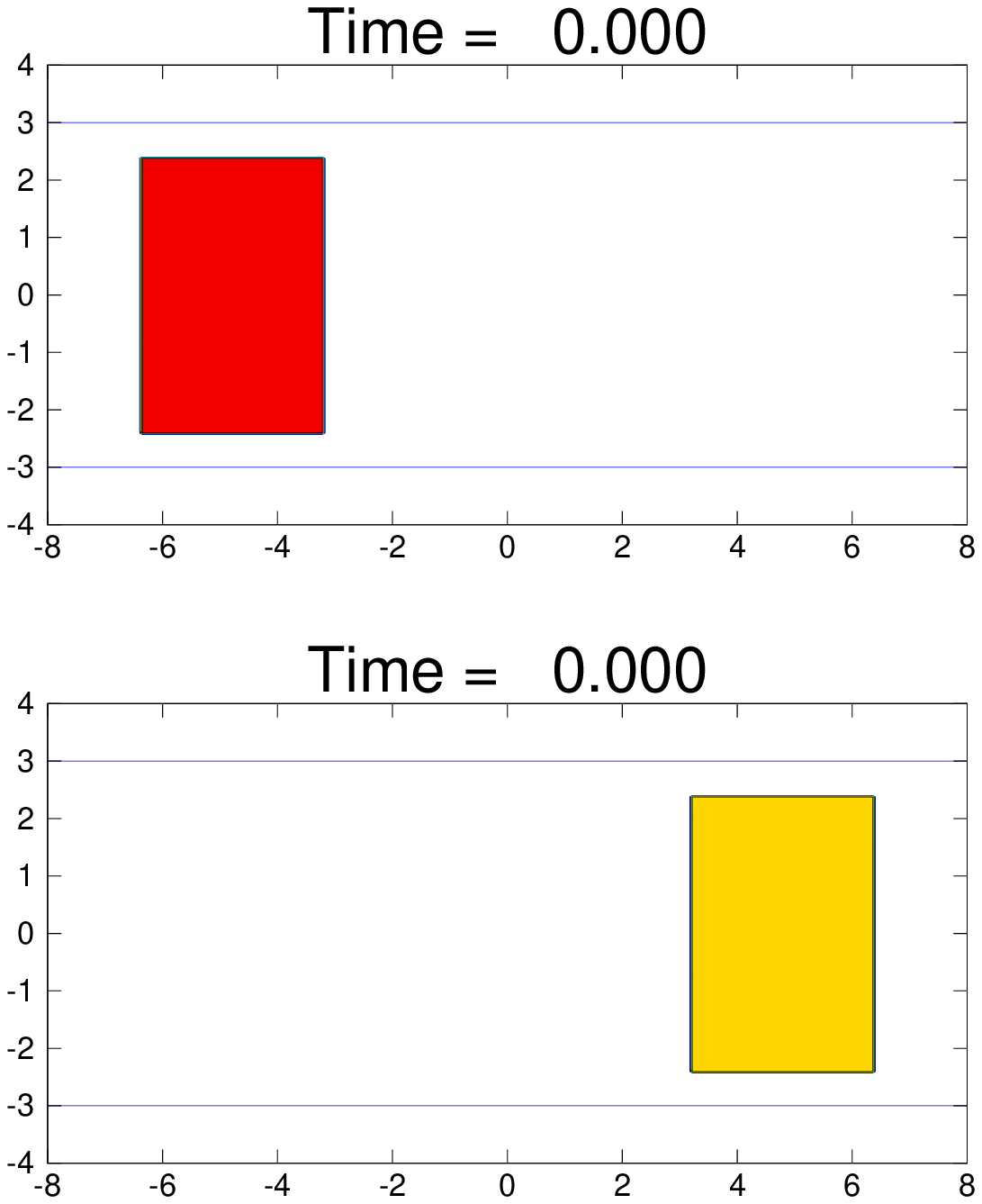}%
  \includegraphics[width=0.25\textwidth, trim=120 40 121 30]{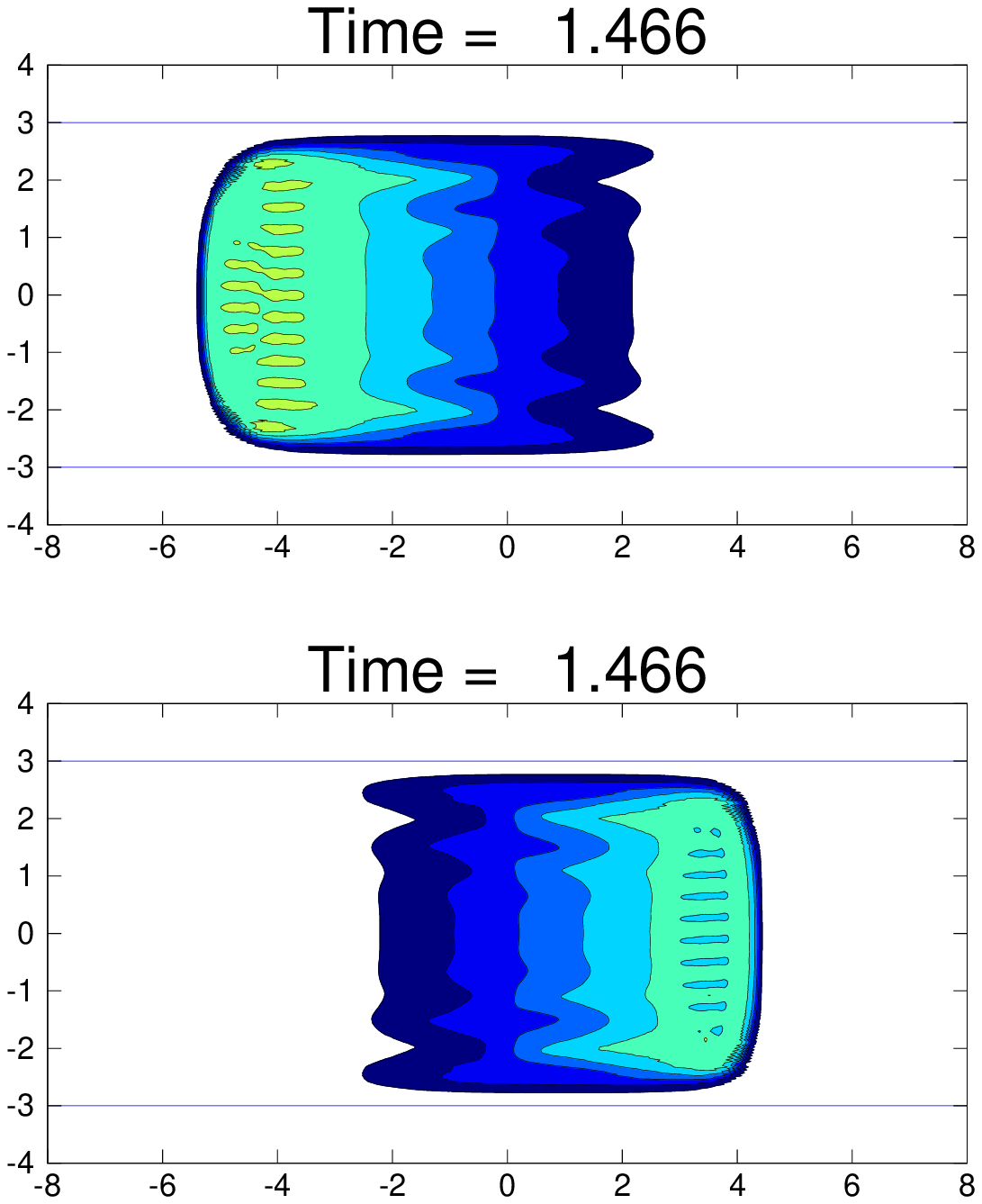}%
  \includegraphics[width=0.25\textwidth, trim=120 40 121 30]{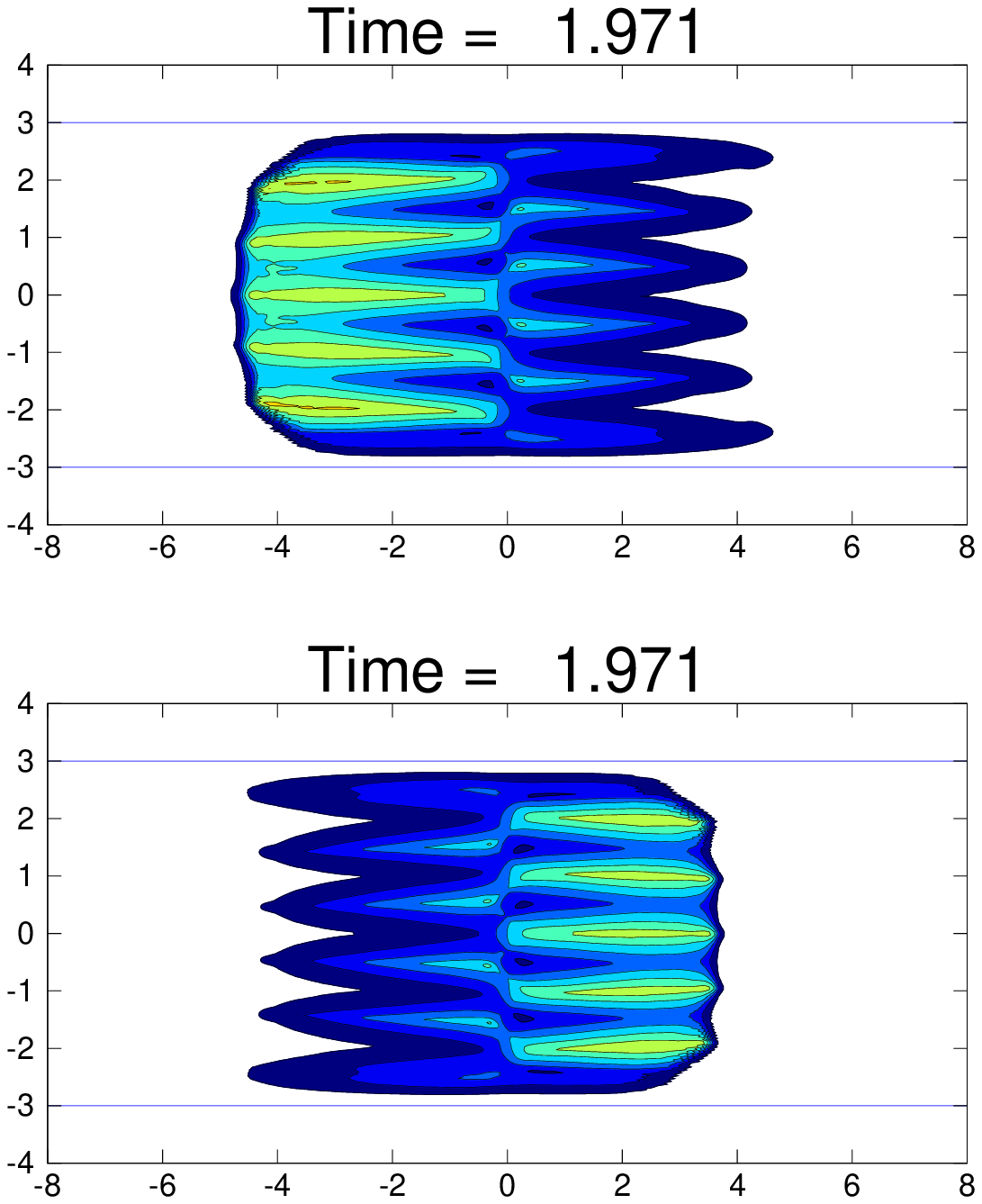}%
  \includegraphics[width=0.25\textwidth, trim=120 40 121 30]{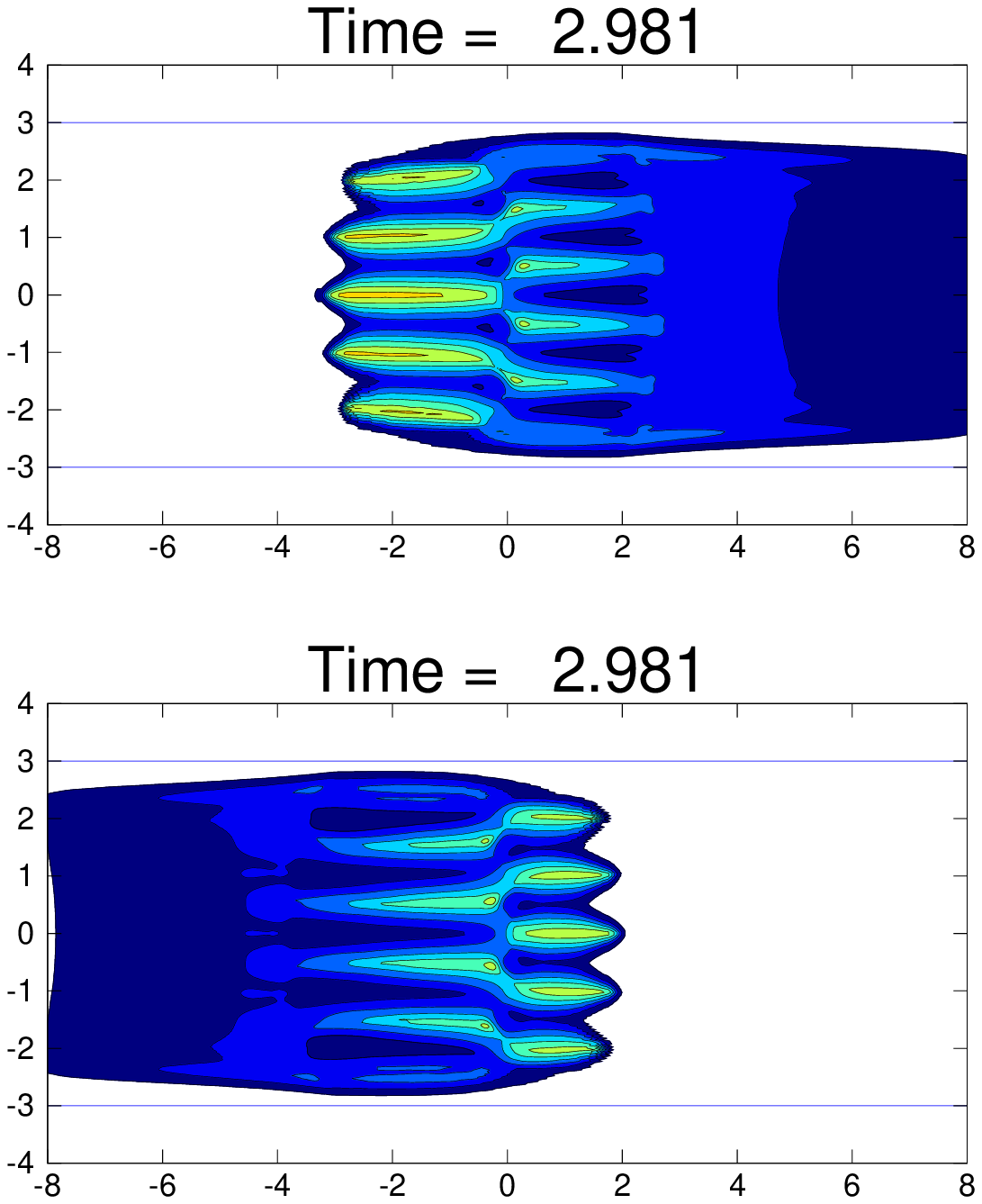}%
  \caption{Numerical integration of~\eqref{eq:order2}. Above, the population $\rho_1$
    moving to the right and, below, $\rho_2$ moving to the left. Note
    the lanes that are formed. First, due to the self-interaction ($\epsilon_i =0.3$)
    within each populations and then due to the crossing between the
    two populations ($\epsilon_o=0.7$). The latter lanes of different populations do not
    superimpose. Picture from \cite{ColomboLecureux}.}
  \label{fig:NI1}
\end{figure}

Another  situation developping interesting features is the following: two populations  
are initially uniformly distributed in the same region. At time t =0, the first populations  
starts moving towards an exit (first line of Figure \ref{fig:NI}), on the right while the second moves only to let 
the first one pass (second line of Figure \ref{fig:NI}).
\begin{figure}[htpb]
  \centering%
  \includegraphics[width=0.25\textwidth, trim=120 40 121
  30]{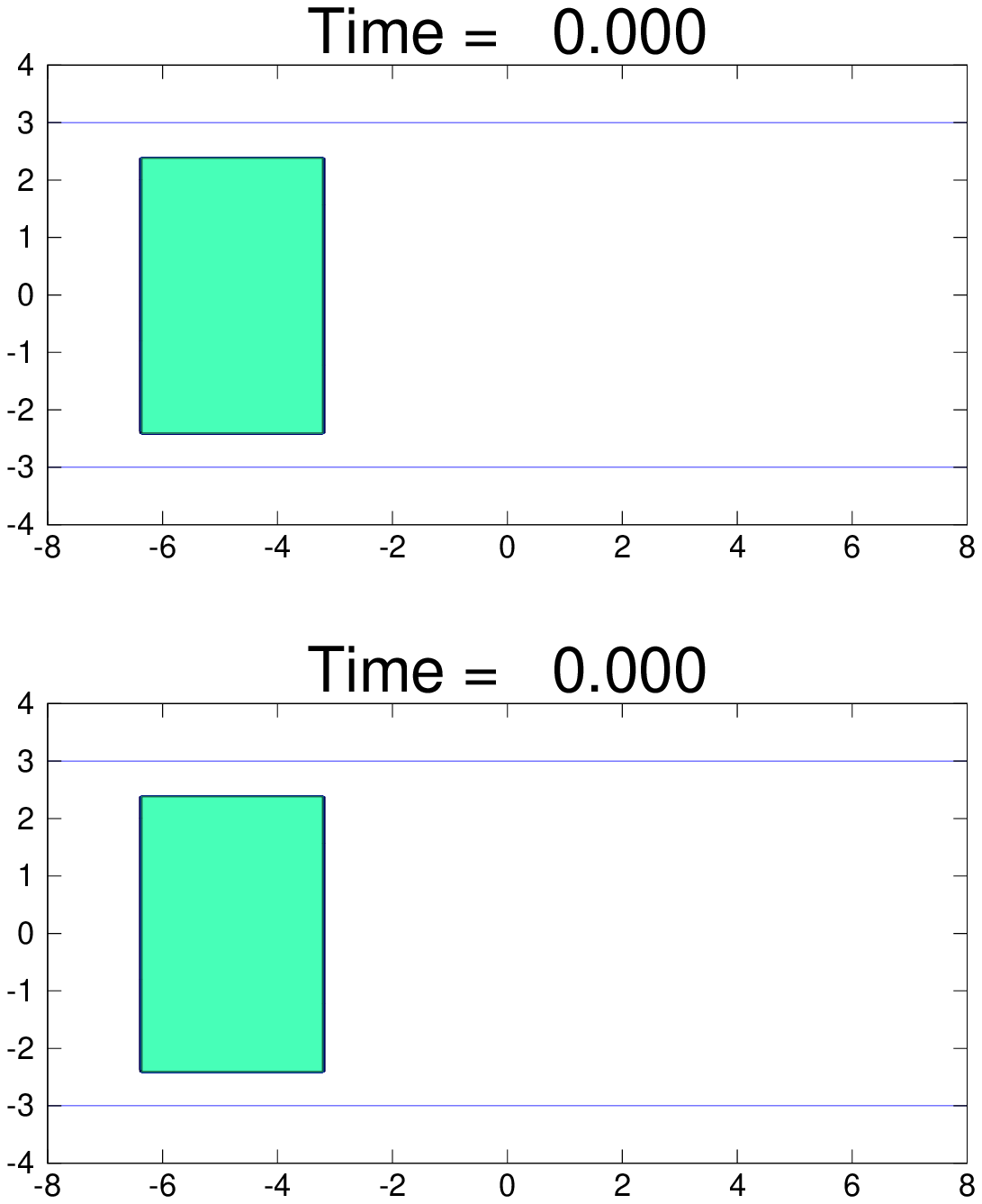}%
  \includegraphics[width=0.25\textwidth, trim=120 40 121
  30]{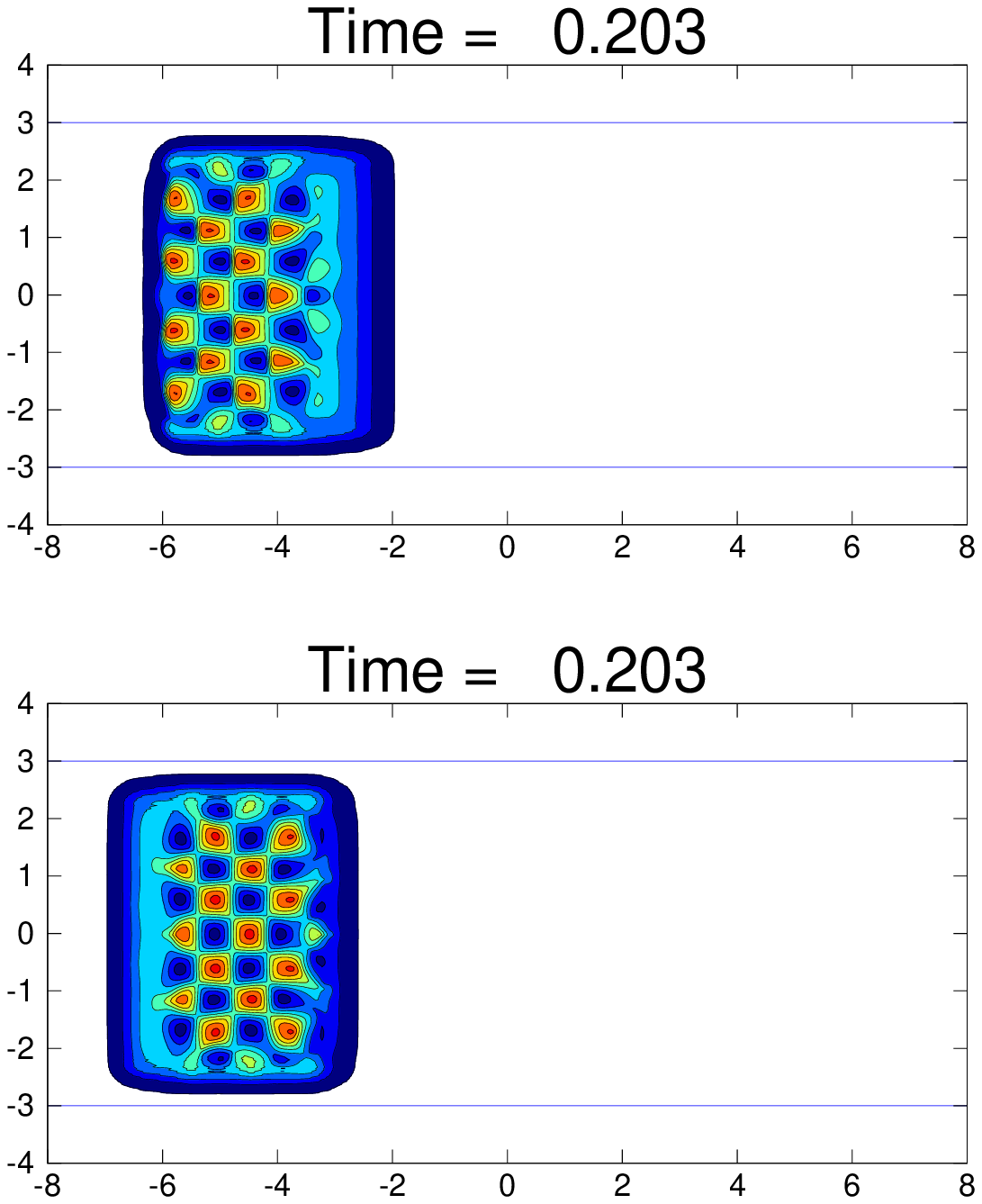}%
  \includegraphics[width=0.25\textwidth, trim=120 40 121
  30]{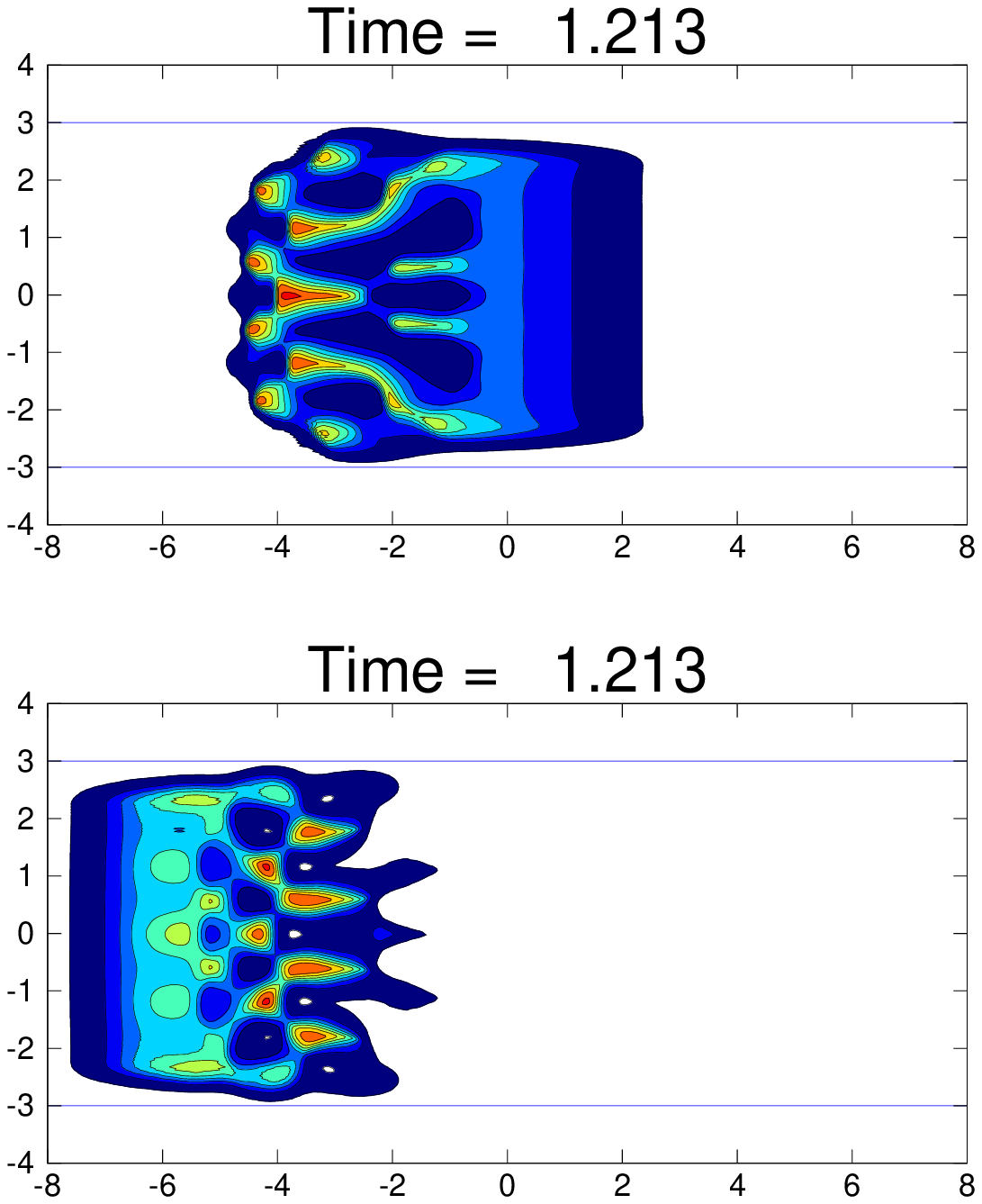}%
  \includegraphics[width=0.25\textwidth, trim=120 40 121
  30]{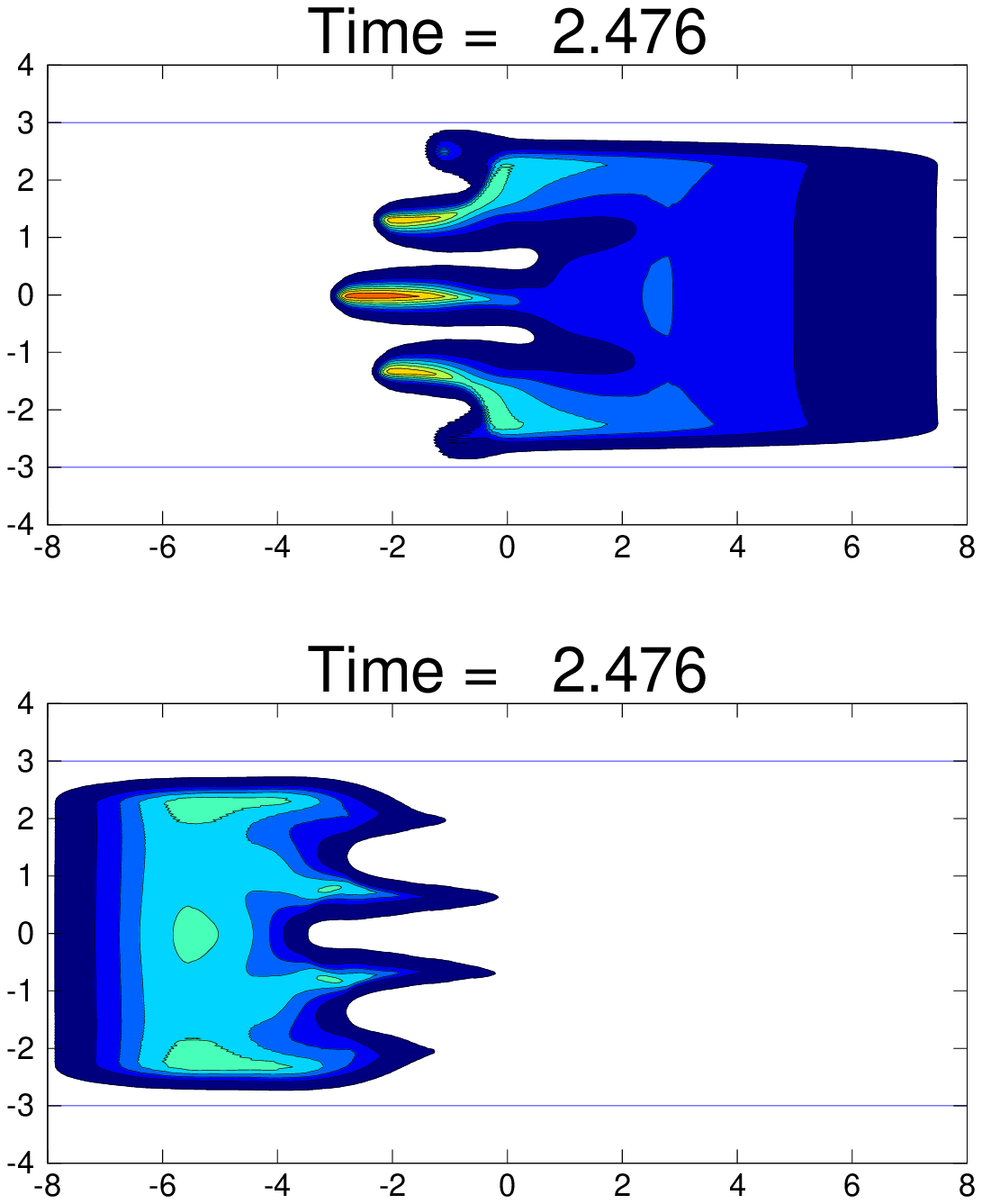}%
  \caption{Numerical integration of~\eqref{eq:order2} when with $\epsilon_i=0$, $\epsilon_o=0.3$. Above, the
    population $\rho_1$ and, below, $\rho_2$. Note first the formation
    of small clusters separating the two populations, then lanes and,
    finally, a sort of fingering. Picture from \cite{ColomboLecureux}.}
  \label{fig:NI}
\end{figure}

\subsubsection{Interaction group / isolated agent}
Let us finally introduce a model that would take into account the situation in which an isolated individual interact with a crowd. We think for example to the situation in which a predator is running after a group of preys. We can also think to the case of a leader willing to carry a group of followers to a given region.  Let $\rho\in \rpic$  be the density of the group and   $p\in \reali^{k}$ be  the position of an isolated agent (\textit{e.g.}  a leader or a predator). We describe the interaction by the coupling (see  \cite{ColomboMercier}):
\begin{equation}\label{eq:groupindiv}
  \left\{
    \begin{array}{l}
      \partial_t \rho
      +
      \Div  \big(\rho \,V \left(t, x, \rho, p(t)\right)\big)
      =
      0\,,
      \\
      \dot p
      =
      \phi\left( t, p, \left(A \rho(t)\right) (p(t)) \right)\,,
    \end{array}
  \right.
  \qquad
  \begin{array}{rcl}
    (t, x )& \in & \reali^+ \times \reali^{N}\,,
  \end{array}
\end{equation}
with initial conditions
$$
  \rho(0,x) =  \rho_0(x)\,,
      \qquad\qquad
      p(0) = p_0\,.
$$
Using once again Kru\v zkov theory and tools on the  stability of ordinary differential equations,  in collaboration with R. M. Colombo, we proved existence and uniqueness of solutions (see \cite[Theorem 2.2]{ColomboMercier}).
%\begin{theorem}[see \cite{ColomboMercier}]\label{thm:groupindiv}
%Assume $V$, Let $\rho_0\in , p_0\in \reali^N$, then there exists a unique weak entropy solution to (\ref{eq:groupindiv}).
%\end{theorem}

For example, we can consider:
\begin{itemize}
\item Followers / Leader (see Figure \ref{fig:hawk}): here  the vector $p\in \reali^2$ is the position of the leader and $\rho$ is the density of the group of followers.  The function $v(\rho)$ describes essentially the speed of the followers and is, as usual, a decreasing function, vanishing in $\rho=1$;   the direction of a follower located in $x$ is given by the vector $p(t)-x$, directed toward the leader. The velocity of the leader is increasing  with respect to the averaged density $\rho*\eta$, computed in the position of the leader. Indeed, we expect the leader to wait for the followers to join him when the density of followers is small around him, and to accelerate when the density of followers becomes bigger; this means low speed when the averaged density is small and high speed when the averaged density is maximal. The  direction  of the leader is a chosen vector field $\vec \psi(t)$. 
\[
  \left\{
    \begin{array}{l}
      \partial_t \rho
      +
      \Div\left(\rho \, v(\rho) \, (p(t)-x)e^{-\norma{p-x}}\right)
      =
      0\,,
      \\[5pt]
      \dot p
      =
      (1+ \rho*\eta(p(t))\, ) \; \vec{{\psi}}(t)\,.
    \end{array}
  \right.
\]
\begin{center}
\begin{figure}[htpb]
  \centering
  \includegraphics[width=0.25\textwidth]%, trim=75 120 20 120]
  {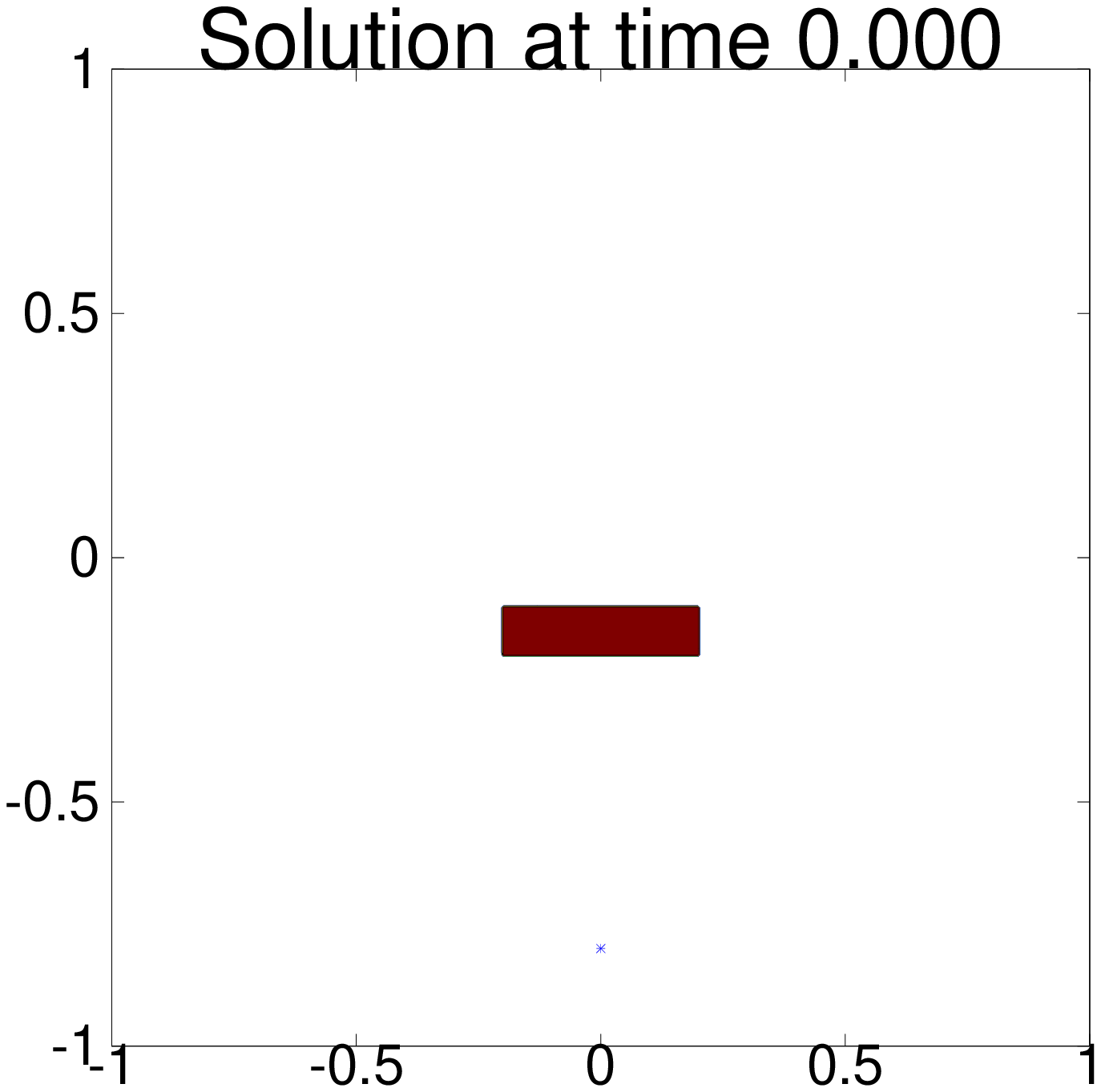}
  \includegraphics[width=0.25\textwidth]
  {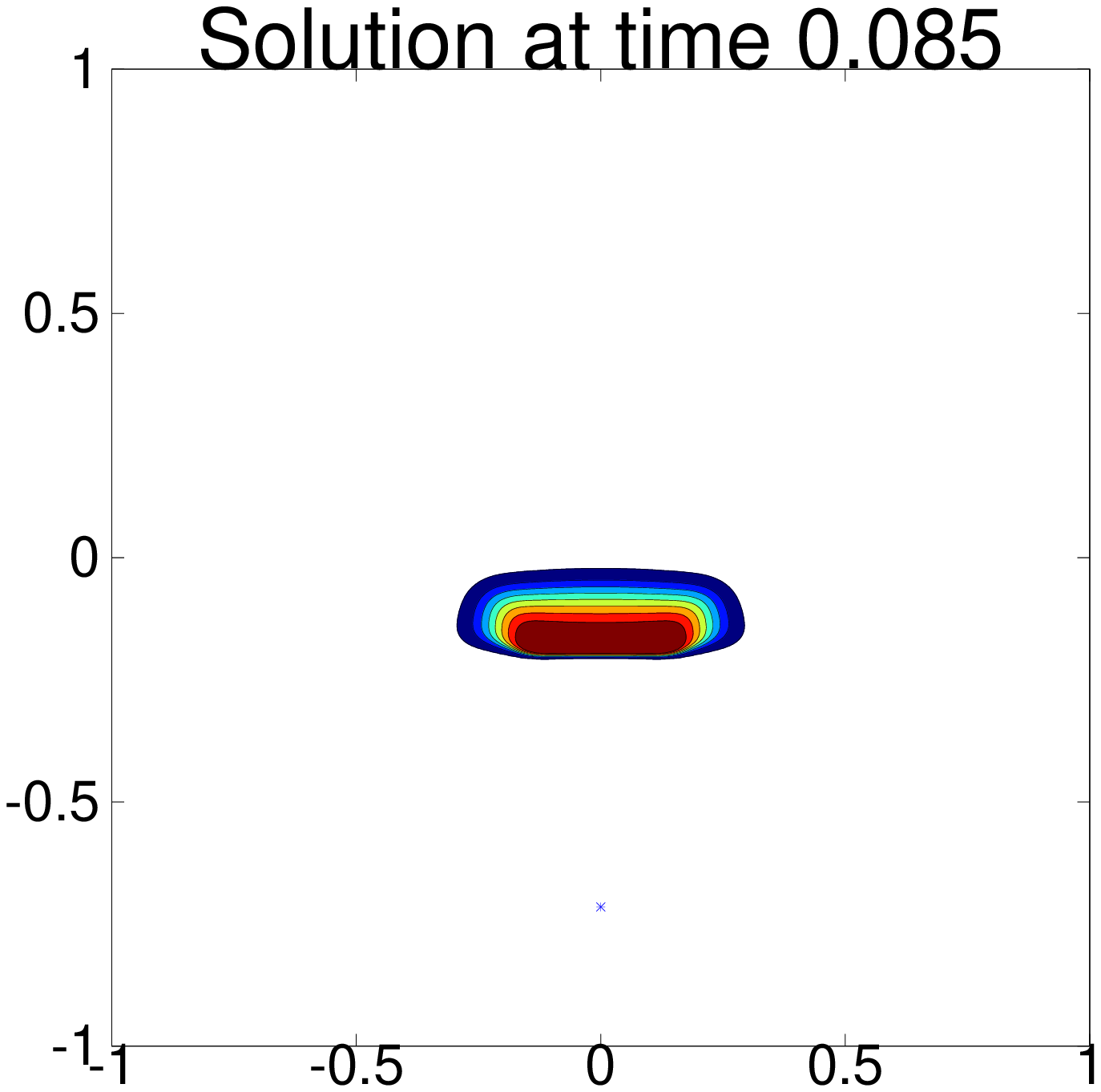}
  \includegraphics[width=0.25\textwidth]{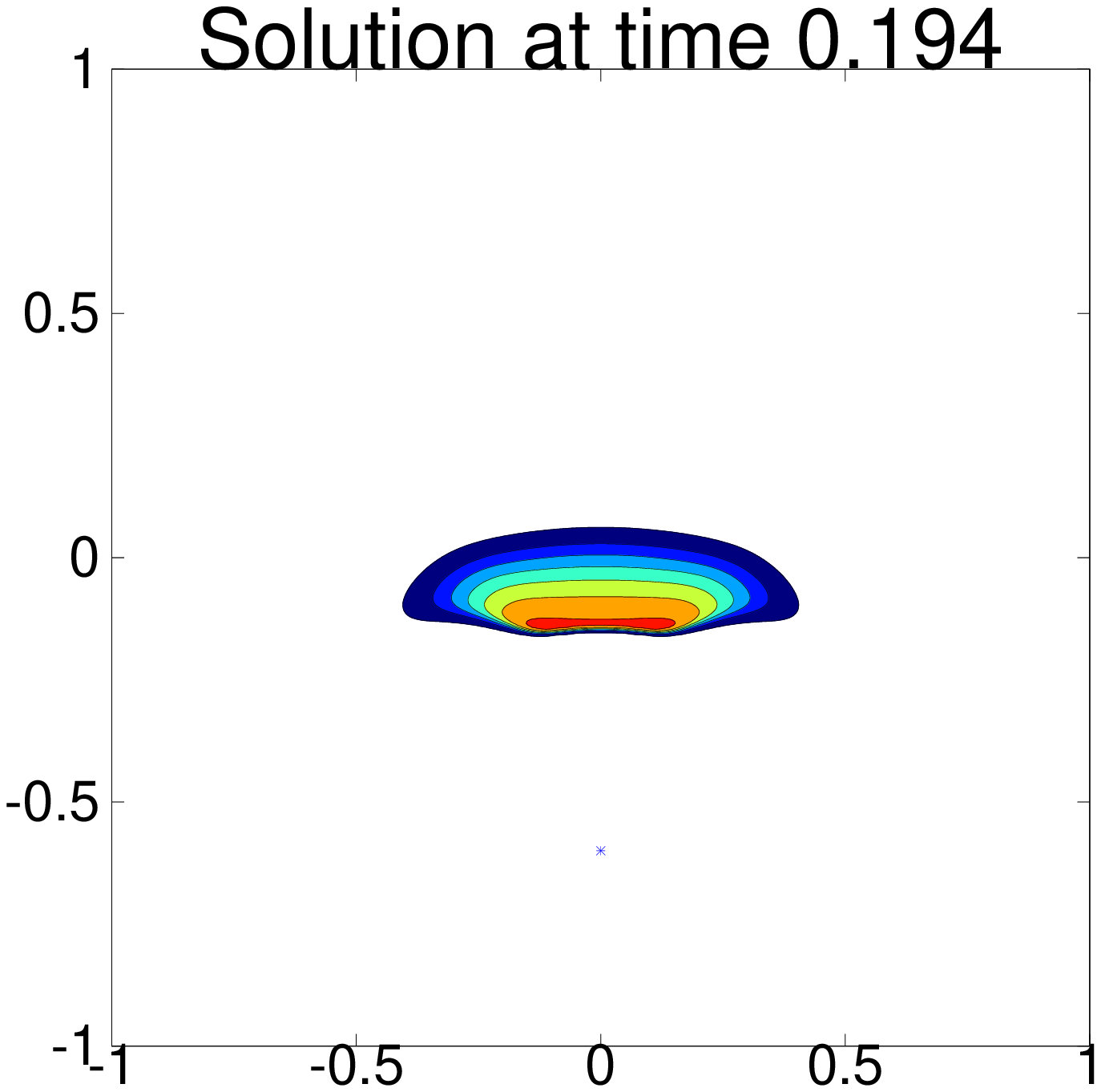} \\[0.7cm]
  \includegraphics[width=0.25\textwidth]{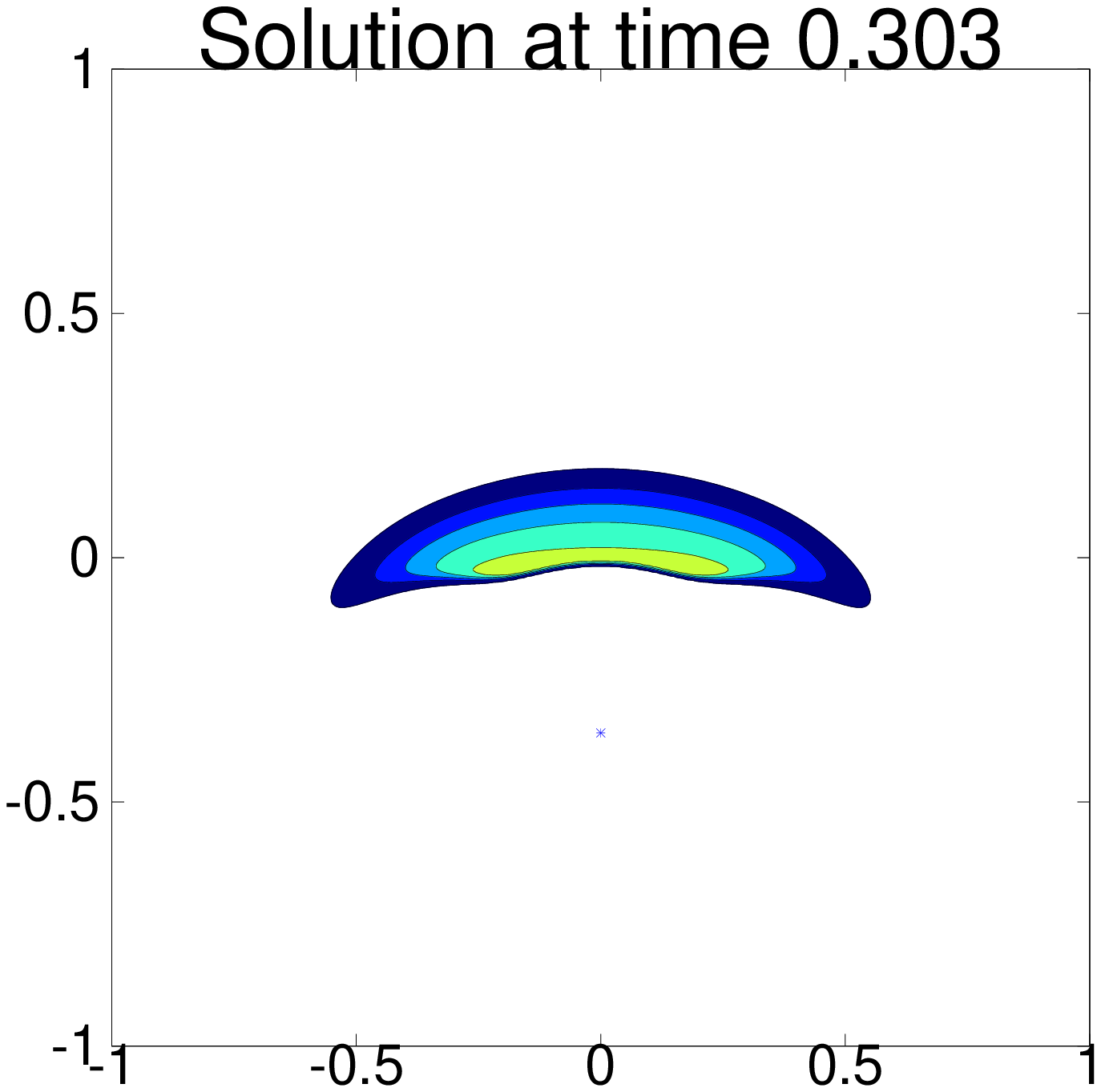}
  \includegraphics[width=0.25\textwidth]{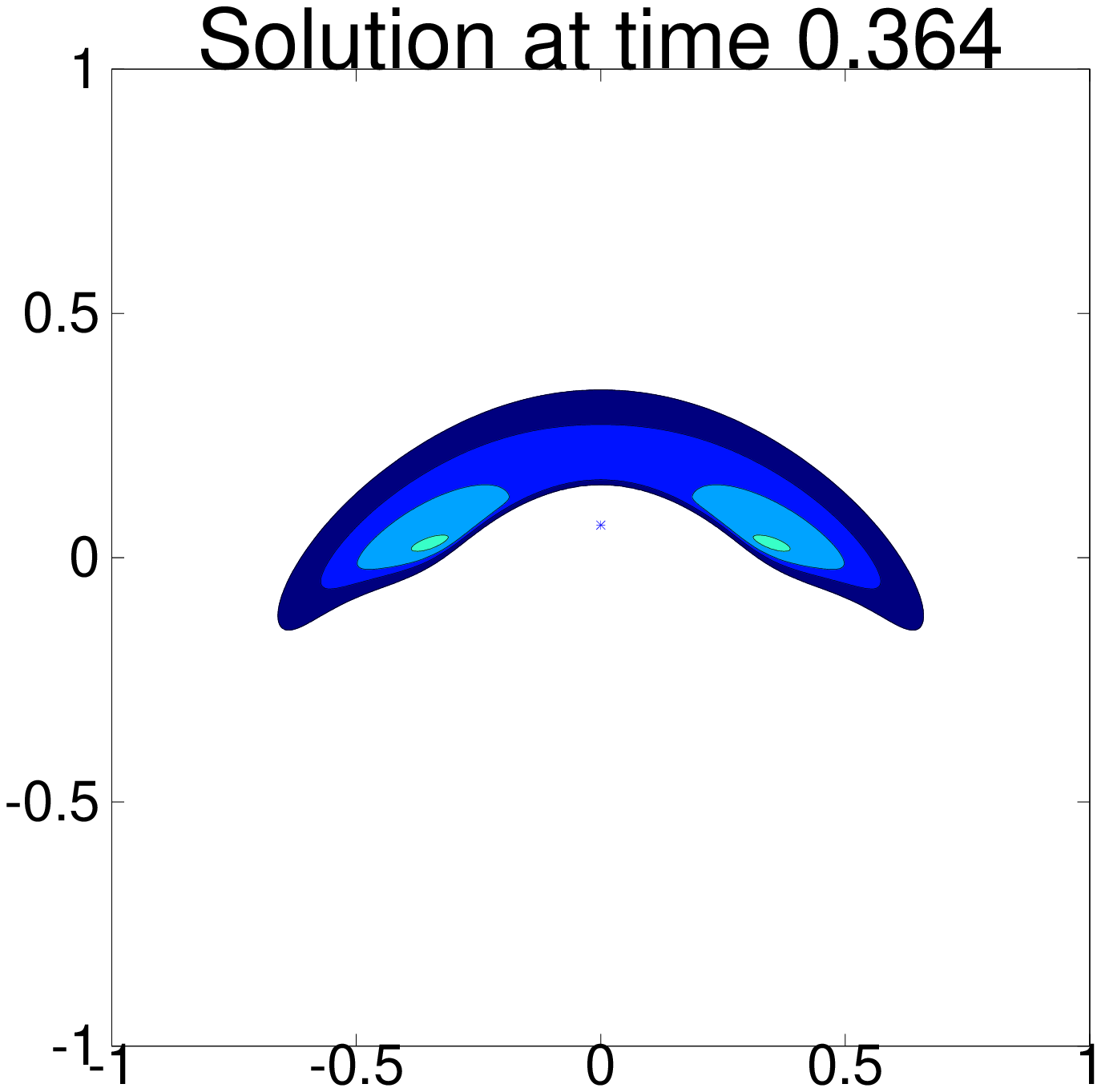}
  \includegraphics[width=0.25\textwidth]{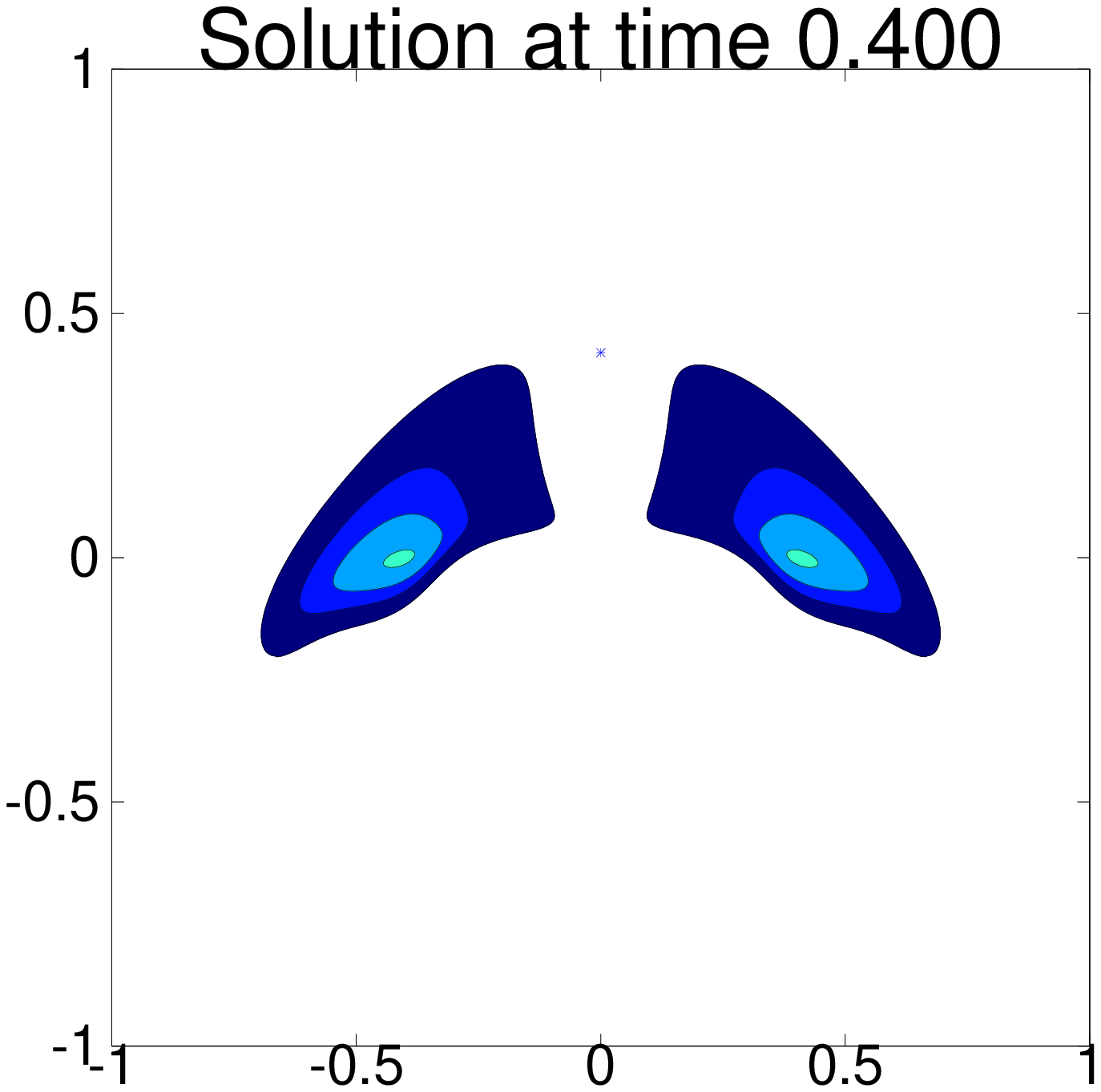}%
\caption{Numerical solution of (\ref{eq:groupindiv}) in a predator/prey interaction. Picture from \cite{ColomboMercier}.}  \label{fig:hawk}
\end{figure}
\end{center}
\item Sheeps / Dogs: in this example, there are $n$ dogs, located in $p_i(t)\in \reali^2$ for $i\in\{1, \ldots, n\}$ and a group of sheeps of density $\rho(t,x)$. The dogs want to constrain the sheeps to a given area by running around them and the sheeps are afraid by the dogs. As above, the speed of the sheeps is given by the decreasing function $v(\rho)$ that vanishes in $\rho=1$. The direction of a sheep located in $x$ is a sum of two terms. The first one is their prefered direction $\vec \nu(x)$ ; the second one is given by the vector   $\sum_{i=1}^n (x-p_i(t)) e^{-\norma{p_i-x}}$ representing the repulsive effects of the dogs on the sheeps, and its presence is is less felt if it is too far away. Each dog run around the flock tacking the direction perpendicular to the direction of maximal averaged density.
\[
  \left\{
    \begin{array}{l}
      \partial_t \rho
      +
      \Div\left( \rho \, v(\rho) \left( \vec{\nu}(x)
      + \sum_{i=1}^n  (x-p_i)e^{-\norma{p_i-x}} \right)\right)
      =
      0\,,
      \\[5pt]
 \displaystyle     \dot p_i
      =
     \frac{(\rho *\nabla \eta )^\perp(p_i(t))}{\sqrt{1+\norma{\rho *\nabla \eta(p_i(t)) }^2}}\,, \qquad \textrm{for all } i\in \{1, \ldots n\}\,.
    \end{array}
  \right.
\]
\item Preys / Predator: here a predator, located in $p(t)\in \reali^3$ is pursuing a group of preys of density $\rho(t,x)$. For example, we can think to an hawk pursuing a swarm of little birds or to a shark pursuing a shoal of little fishes. As above, the speed of the preys  is given by the decreasing function $v(\rho)$ that vanishes in $\rho=1$. The direction of a prey located in $x$ is given by the vector $(x-p(t)) e^{-\norma{p(t)-x}} $ so that the predator has a repulsive effect on the preys, and its presence is is less felt if it is too far away. The acceleration of the predator is directed toward the maximal averaged density of preys, as felt from its position, so it is directed as $\nabla (\rho*\eta)(p(t))$.
\[
  \left\{
    \begin{array}{l}
      \partial_t \rho
      +
      \Div\left(  \rho\, v(\rho) \, 
  \left(
   1
    +
   \, e^{-\norma{x-p(t)}} \, \left(x-p(t)\right)
  \right)\right)
      =
      0\,,
      \\[5pt]
     \displaystyle \frac{\mathrm{d}^2 p}{\mathrm{d}t^2}
      =
    (\rho *\nabla )\eta(p(t))\,.
    \end{array}
  \right.
\]
\end{itemize}

%
%
%\begin{center}
%\begin{figure}[htpb]
% \centering
%  \includegraphics[width=0.33\textwidth, trim=75 120 20
%  120]{flute1001ct.eps}
%  \includegraphics[width=0.33\textwidth, trim=75 120 20
%  120]{flute1018ct.eps}
%  \includegraphics[width=0.33\textwidth,
%  trim=75 120 20 120]{flute1055ct.eps} \\[0.7cm]
%  \includegraphics[width=0.33\textwidth, trim=75 120 20
%  120]{flute1095ct.eps}
%  \includegraphics[width=0.33\textwidth, trim=75 120 20
%  120]{flute1145ct.eps}
%  \includegraphics[width=0.33\textwidth, trim=75 120 20
%  120]{flute1193ct.eps}%
%  \label{fig:lanes}
%\end{figure}
%\end{center}

Note that Theorem \ref{thm:panicOT2} deals with measure solutions. Consequently, if $\rho_1\in \L1$ and $\rho_2=\delta_{p(t)}$, we recover the above coupling PDE/ODE. However, in Theorem  \ref{thm:panicOT2}, there are restricitions on the flow as we only deal with continuity equations. In particular, with a well-chosen nonlinearity with respect to $\rho$ in the first equation of (\ref{eq:groupindiv}), we are able to provide for (\ref{eq:groupindiv}) an a priori uniform bound on the $\L\infty$ norm, which is not possible for (\ref{eq:panic2}): in the three examples above, if $v(1)=0$, then  $0\leq \rho_0\leq 1$ implies  $0\leq \rho (t)\leq 1$ for all $t\geq 0$.

\section{Using Kru\v zkov theory}\label{sec:k}
\subsection{General Theory on scalar conservation laws}
In this section, we consider classical \emph{scalar} conservation laws. 
\begin{equation}\label{eq:scl}
\pt_t u+\Div f(t,x, u)=F(t,x,u)\,,\qquad \qquad u(0, x)=u_0(x)\,,
\end{equation}
where $f$ is the flow and $F$ is the source. Note that, $\Div$ stands for the total divergence whereas $\div$ stands for the partial divergence. Thus $\Div f (t,x, u(t,x))=\div f(t,x,u(t,x))+\pt_u f(t,x,u(t,x)) \cdot \nabla u(t,x)$.

Let us recall the definition of weak entropy solution:
\begin{definition}\label{def:wes}
The function $\rho\in \L\infty([0,T]\times \reali^N, \reali)$  is called \emph{weak entropy solution} to the Cauchy problem (\ref{eq:scl}) if it satisfies for all $k\in \reali$ and any test-function $\phi\in \Cc0(]-\infty,T[\times \reali^N, \rpic)$,
 \begin{equation}
    \begin{array}{r}
      \displaystyle
      \int_{\rpis}\int_{\reali^N} 
      \left[
        (u-k) \, \pt_t \phi
        +
        \left( f(t,x,u) - f(t,x,k) \right) \nabla_x \phi
        +
        \left(F(t,x,u)-\div f(t,x,k)\right)\phi
      \right]\\
     \times
      \mathrm{sign}(u-k) \,
      \mathrm{d}x \, \mathrm{d}t
      +\int_{\reali^N} \modulo{\rho_0(x)-k}\phi(0,x)\d{x}
      \geq
      0\,.
    \end{array}
  \end{equation}
\end{definition}
We also recall the well-known Kru\v zkov theorem \cite[Theorem  4]{Kruzkov}, giving existence and uniqueness of weak entropy solutions.
\begin{theorem}[Kru\v zkov]\label{thm:kru}
Let $u_0\in (\L\infty\cap\L1)(\reali^N, \reali)$.   For all $A, T\geq 0$,
we denote  $\Omega_{T}^A=[0,T]\times\reali^N\times[-A,A]$ and $\Omega=\rpic\times\reali^N\times \reali$. Under the conditions $f\in \C0(\Omega;\reali^N)$, $F\in \C0(\Omega;\reali)$ and
\[
\begin{array}{cl}
\mathbf{(K)}&
\left\{\begin{array}{l}
f \in \C0(\Omega; \reali^N)\,,
        \qquad 
        F \in \C0(\Omega; \reali)\,,\textrm{ and for all } T,  A>0:
        \\[5pt]
f\,,\; F \;\textrm{ have continuous derivatives: }\;\pt_u f\,,\; \pt_u\nabla f\,,\; \nabla^2 f\,,\; \pt_u F\,,\; \nabla F\,,\\
\textrm{for all } T, A>0\;,\;\; \pt_u f\in \L\infty(\Omega_{T}^A)\,,\\
 F-\div f\in \L\infty(\Omega_T^A)\,,\qquad \; \pt_u(F-\div f)\in \L\infty(\Omega_T^A)
 \end{array}\right.
 \end{array}
\] 
there exists a unique weak entropy solution $u\in \L\infty([0,T];\L1(\reali^N;\reali))$ of (\ref{eq:scl}) that is right-continuous in time.

Let  $v_0\in (\L1\cap\L\infty)(\reali^N;\reali)$. Let $u$ be the solution associated to the initial condition $u_0$ and $v$ be the solution associated to the initial condition $v_0$. Let $M$ be such that $M\geq \sup(\norma{u}_{\L\infty(\rpic \times \reali^N;\reali)}, \norma{v}_{\L\infty([0,T]\times \reali^N;\reali)})$. Then, for all $t\in [0,T]$, with $\gamma=\norma{\pt_u F}_{\L\infty(\Omega_T^M)}$, we have
\begin{equation}\label{eq:kru}
\norma{{(u-v)(t)}}_{\L1} \leq e^{\gamma t} \norma{u_0-v_0}_{\L1}\,.
\end{equation}
\end{theorem}

\subsubsection{Estimate on the dependence with respect to flow and source}

Let us now consider a pair of different flows $f,g\in \C2 ([0,T] \times \reali^N \times \reali; \reali^N)$ and different sources $F,G\in \C1 ([0,T] \times \reali^N \times \reali; \reali)$, let us denote $u,v$ the solutions of
\begin{align}
      \partial_t u + \Div \, f(t,x,u)& = F(t,x,u)\,, &u(0, \cdot)&=u_0\,,\label{eq:u}
      \\
       \partial_t v + \Div \, g(t,x,v) &=  G(t,x,v)\,, &v(0, \cdot)&=v_0\,,\label{eq:v}
\end{align}
with initial conditions   $u_0, v_0\in (\L1\cap\L\infty \cap\BV)(\reali^N, \reali)$. We want here to  estimate $\norma{(u-v)(t)}_{\L1}$ with $f-g$, $F-G$, $u_0-v_0$. To do that, we use the doubling variable method due to Kru\v zkov \cite{Kruzkov}. This strategy was already employed  by Lucier \cite{lucier}  and Bouchut \& Perthame \cite{BouchutPerthame} to study the case  in which  the flows $f, g$ depend only on $u$ and the source are identically zero ($F=G=0$). 
A key ingredient in the previous results is that the solution of $\pt_t u+\Div (f(u))=0$ with initial condition $u_0\in (\L\infty\cap\L1\cap\BV)(\reali^N, \reali)$ satisfies $\tv(u(t))\leq \tv (u_0)$. %However, this estimate on the total variation is no longer valid in the case $f$ and $F$ depend on the three variables $t, x, u$. Consequently, our first step is to find a estimate on the total variation.

Let us remind the definition of total variation.
\begin{definition}
\label{def:tv}
For $u\in \Lloc1(\reali^N;\reali)$ we denote the total variation of~$u$: 
\begin{eqnarray*}
\tv(u)&=& \mathrm{sup}\big\{\int_{\reali^N} u\, \div \Psi\, ;\quad \Psi\in \Cc1(\reali^N;\reali^N)\, ,\quad \norma{\Psi}_{\L\infty}\leq 1\big\}\,.
\end{eqnarray*}

The space of function with bounded variation  is then defined as
\[
\BV(\reali^N;\reali) =\left\{u\in \Lloc1 ; \tv(u)<\infty\right\}\,.
\]
\end{definition}

\begin{remark}
If $u\in (\C1\cap\W11)(\reali^N, \reali)$ then $\tv(u)=\norma{\nabla u}_{\L1}$.
When  $f $ and  $F$  are not depending on  $u$, we have
\[
u_0\in (\L\infty \cap\BV)(\reali^N, \reali) \Rightarrow \forall t\geq0\,,\quad u(t)\in (\L\infty\cap\BV)(\reali^N, \reali)
\]
and, with {$\gamma=\norma{\pt_u F}_{\L\infty(\Omega_T^M)}$},
\[
\tv(u(t))\leq \tv(u_0)e^{\gamma t}\,.
\]
\end{remark}

\begin{theorem}[see \cite{lucier}]
Let $f,g :\reali\to \reali^N$ be globally  Lipschitz, let $u_0,v_0\in (\L1\cap\L\infty)(\reali^N;\reali)$  be initial conditions of
\[
\pt_t u+\Div\, f(u)=0\,,\qquad \pt_t v+\Div\, g(v)=0\,.
\]
Assume  furthermore  that $v_0\in \BV(\reali^N;\reali)$. Then  for all $ t\geq 0$,
\[
\norma{(u-v)(t)}_{\L1}\leq \norma{u_0-v_0}_{\L1} +C\, t\, \tv(v_0)\,\Lip(f-g)\,,
\]
where $C>0$ is a given constant.
\end{theorem}

We want to generalize this theorem. 
Let us introduce the set of hypotheses:
\begin{align*}
    \mathbf{(FS)}&
    \left\{
      \begin{array}{l}
      \textrm{for all }U,\, T>0 \,,\qquad \pt_u F\in \L\infty(\Omega_T^U)\\[5pt]
        \displaystyle
%       \textrm{for all }U,\,T>0\,,\quad 
\int_{0}^T \!\! \int_{\reali^N} \!
        \norma{(F- \div\, f)(t, x, \cdot)}_{\L\infty({[-U,U]};\reali)} 
        \, \mathrm{d}x\, \mathrm{d}t < + \infty\,.
      \end{array}
    \right.
    \end{align*}

Denoting $\Omega_T^V=[0,T]\times \reali^N\times [-V,V]$, we obtain:
\begin{theorem}[see \cite{Lecureux}]\label{thm:stab} 
Assume that  $(f,F), (g,G)$ satisfy \textbf{(K)}, that  $(f-g,F-G)$ satisfies \textbf{(FS)}. 
Let $u_0,v_0\in (\L1\cap\L\infty\cap\BV)(\reali^N;\reali)$. Let  $u$ and $v$ be the solutions of (\ref{eq:u}) and (\ref{eq:v}) associated to  $(f,F)$ and $(g,G)$  with initial conditions $u_0$ and $v_0$. Assume furthermore that $\tv(u(t))<\infty$ for all $t\geq 0$.

Let $V=\max(\norma{u}_{\L\infty}, \norma{v}_{\L\infty})$ and 
{$
\kappa= \norma{\pt_u F}_{\L\infty(\Omega_T^V)}
$}.
Then for all $ t\in [0,T]$:
\begin{align*}
&\norma{(u-v)(t)}_{\L1}\leq e^{\kappa t}\norma{u_0-v_0}_{\L1} +e^{\kappa t}\sup_{\tau\in [0,t]} \left(\tv(u(\tau)) \right) \int_0^t \norma{\pt_u(f-g)(\tau )}_{\L\infty(\reali^N\times[-V;V])}\d{\tau }\\[5pt]
& \quad+\int_0^t e^{\kappa (t-\tau)}\int_{\reali^N} \norma{((F-G)-\div\,(f-g))(\tau,x,\cdot)}_{\L\infty({[-V,V]})} \mathrm{d}x \mathrm{d}\tau\,.
\end{align*}
\end{theorem}
The proof of this theorem is based on the  Kru\v zkov doubling variables method and is detailed in \cite{Lecureux}. 

In some particular cases, we recover known estimates:
\begin{itemize}
\item $f(u), g(u),  F=G=0$ : $\kappa=0$ and 
\begin{align*}
&\norma{(u-v)(t)}_{\L1}\leq \norma{u_0-v_0}_{\L1} +t\,\tv(u_0) \norma{\pt_u(f-g)}_{\L\infty({\Omega_T^V})}\,.
\end{align*}
\item $f(t,x), F(t,x)$ : $\kappa=0$, $u(t,x)=u_0(x)+\int_0^t(F-\div f)(\tau, x)\d{\tau}$,  and 
\begin{align*}
\norma{(u-v)(t)}_{\L1}\leq &\norma{u_0-v_0}_{\L1} +\int_0^t \int_{\reali^N} \modulo{((F-G)-\div\,(f-g))(\tau,x)} \mathrm{d}x \mathrm{d}\tau\,.
\end{align*}
\end{itemize}

\subsubsection{Total variation estimate}
To complete Theorem \ref{thm:stab}, we need an estimate on the total variation, defined in  Definition \ref{def:tv}. To do so we need a few more hypotheses. Indeed, there is no reason why the total variation should be bounded for all time. In fact, let  us consider the equation $\pt_t u+\pt_x \cos x=0$, $u_0=0$. The solution of this Cauchy problem is $u(t,x)=t\sin x$ whose total variation is 0 at time 0 and $+\infty$ for any time $t>0$.

Our goal now is to give a general estimate on the total variation.
Let us introduce the following set of hypotheses
\begin{align*}
 \mathbf{(TV) :}\left\{
      \begin{array}{l}
     \textrm{for all }A, T>0\\
\nabla\pt_u f\in \L\infty(\Omega_T^A  ;\reali^{N\times N})\,,\quad \pt_u F\in \L\infty(\Omega_T^A  ;\reali)\,,\\[5pt]
% \textrm{for all } t>0\,,\; u\in \reali\,,\quad\nabla^2 f(t,\cdot,u) \in \L1(\reali^N;\reali^{N\times N\times N})\,,\\[5pt]
\displaystyle  \int_0^T\!\!\int_{\reali^N}\!\!\norma{\nabla(F-\div f)(t,x,\cdot)}_{\L\infty([-A,A];\reali^N)} < \infty\,.
      \end{array}
    \right.
\end{align*}

We obtain
\begin{theorem}
  \label{teo:tv}
 Let us assume $(f,F)$ satisfies \textbf{(K)}-\textbf{(TV)}. Denote $
  W _N
  =
  \int_0^{\pi/2} (\cos \theta)^N \, \mathrm{d}\theta $,  $M=\norma{u}_{\L\infty([0,T]\times \reali^N)}$, and 
\[
 \kappa_0= (2N+1)\norma{\nabla\pt_u f}_{\L\infty(\Omega_T^M)}+\norma{\pt_u F}_{\L\infty(\Omega_T^M)}\,.
\]
 Then, the weak entropy solution
  $u$ of~(\ref{eq:scl}) satisfies $u(t) \in \BV(\reali^N; \reali)$
  for all $t > 0$, 
 and   
  \begin{eqnarray}
    \label{result}
    \tv \left( u(T) \right)
    &\leq &
    \tv(u_0) \, e^{\kappa_0 T}   +
    N W_N \!\! \int_0^T \!\!e^{\kappa_0(T-t)}\!\! \int_{\reali^N}\!
    \norma{\nabla( F - \div f)(t, x, \cdot)}_{\L\infty([-{U}_t,{U}_t];\reali)}
    \mathrm{d}x\, \mathrm{d}t \,,\nonumber
  \end{eqnarray}
  where ${U}_t= \sup_{y\in \reali^N}\modulo{u(t,y)}$.
\end{theorem}

The proof of this theorem is based on a good representation formula for the total variation and  on the doubling variable method and is detailed in  \cite{Lecureux}.

\begin{remark}\label{rem:opt}
In some cases, we recover  known estimates.  
\begin{itemize}
\item When $f$  depends only on  $u$ and $F=0$, we have a result similar to the one that was already known: $\tv(u(t))\leq \tv(u_0)$.
\item When $f$ and $F$  do not depend on $u$, the equation reduces in fact  to the ODE $\pt_t u =(F-\div f)(t,x)$, whose solution writes 
 \[
 u(t,x)=u_0(x)+\int_0^t(F-\div f)(\tau,x)\d{\tau}\,.
 \] 
 Meanwhile, the bound above reduces to 
 \[
 \tv(u(t))\leq \tv(u_0)+NW_N \int_0^t \int_{\reali^N} \modulo{(F-\div f)(\tau,x)}\d{\tau}
 \]
 which is essentially what we expected, up to the coefficient $NW_N$.
 \end{itemize}
\end{remark}

\begin{proof}
%The proof is based on the Kru\v zkov doubling variable method and on a good representation formula for the total variation. 
Let us admit the following proposition:
\begin{proposition}\label{prop:tv}
Let $\mu \in \Cc\infty(\rpic;\rpic)$ such that $\norma{\mu}_{\L1}=1$ and $\mu'<0 $  on $\rpis$. Let $u\in \Lloc1(\reali^N, \reali)$. We define  $\mu_\lambda(x)=\frac{1}{\lambda^N}\mu\left(\frac{\norma{x}}{\lambda}\right)$. If there exists  $C_0>0$ such that $\forall \lambda>0$, 
\[
\mathcal{I}(\lambda)=\frac{1}{\lambda}\int_{\reali^N}\int_{\reali^N} \modulo{u(x+y)-u(x)}\mu_{\lambda}(y)\mathrm{d}{x}\mathrm{d}{y}\leq C_0,
\] 
then $u\in \BV(\reali^N, \reali)$ and
\[
C_1\,\tv(u)=\lim_{\lambda\to 0}\mathcal{I}(\lambda)\leq C_0.
\]
with $C_1=\int_{\reali^N} \modulo{y_1}\mu(\norma{y})\mathrm{d}{y}$.
\end{proposition}

Hence, we introduce:
\[
\mathcal{F}(T,\lambda)=\int_0^T\int_{\reali^N}\int_{B(x_0,R+M(T_0-t))}  \modulo{u(x+y)-u(x)}\mu_{\lambda}(y)\mathrm{d}{x}\,\mathrm{d}{y}\,\mathrm{d}t\,,
\] 
our goal being to estimate this quantity using the Kru\v zkov doubling variable method.

\textbf{Doubling variable method.}
Let us  denote $u = u(t,x)$ and $v = u(s,y)$ for $(t,x), (s,y) \in \rpis \times
  \reali^N$. For all $k,l \in \reali$ and all test-function
  $\phi = \phi(t,x,s,y)\in\Cc1 \left((\rpis \times \reali^N)^2;
    \rpic \right)$, we have, by definition of weak entropy solution (see Definition  \ref{def:wes})
 \begin{equation}
    \label{eq:utv}
    \begin{array}{r}
      \displaystyle
      \int_{\rpis}\int_{\reali^N} 
      \left[
        (u-k) \, \pt_t \phi
        +
        \left( f(t,x,u) - f(t,x,k) \right) \nabla_x \phi
        +
        \left(F(t,x,u)-\div f(t,x,k)\right)\phi
      \right]
     \\
     \times
      \mathrm{sign}(u-k) \,
      \mathrm{d}x \, \mathrm{d}t
      \geq
      0
    \end{array}
  \end{equation}
 and
  \begin{equation}
    \label{eq:vtv}
    \begin{array}{r}
      \displaystyle
      \int_{\rpis} \int_{\reali^N} 
      \left[
        (v-l) \, \pt_s \phi
        +
        \left( f(s,y,v) - f(s,y,l) \right) \nabla_y \phi
        +
        (F(s,y,v)-\div f(s,y,l))\phi
      \right]
           \\
     \times
      \mathrm{sign}(v-l) \,
      \mathrm{d}y \, \mathrm{d}s
      \geq
      0\,.
    \end{array}
  \end{equation}
Let us choose $k = v(s,y)$ in~(\ref{eq:utv}) and integrate in
  $(s,y)$. Similarly, we choose $l = u(t,x)$ in~(\ref{eq:vtv}) and we integrate in $(t,x)$. 
Furthermore, we choose $\phi(t,x,s,y)=\Psi(t-s, x-y)\Phi(t,x)$ and we summate to get
  \begin{equation}
    \label{eq:sumtv}
   \begin{array}{rc}
  \displaystyle
      \int_{\rpis} \int_{\reali^N}
      \int_{\rpis} \int_{\reali^N} 
      \mathrm{sign}(u-v)
      \bigg[
      (u-v) \, \Psi \, \pt_t \Phi +
      \left( f(t,x,u) - f(t,x,v) \right) \cdot
      \left( \nabla \Phi \right) \Psi &
      \\
      +
      \left( f(s,y,v) - f(s,y,u) - f(t,x,v) + f(t,x,u) \right) \cdot
      \left( \nabla \Psi \right) \Phi&
      \\
      +
      \left(F(t,x,u) - F(s,y,v) + \div f(s,y,u) - \div f(t,x,v) \right) \phi
      \bigg]
      \mathrm{d}x \, \mathrm{d}t \, \mathrm{d}y \, \mathrm{d}s
      &\geq
      0 \,.
    \end{array}
  \end{equation}
We choose now $\mu$ and $\nu$ unit approximations and 
$$\Psi(t-s,x-y)=\frac{1}{\lambda^N}\mu\left(\frac{x-y}{\lambda}\right)\,\frac{1}{\eta} \nu\left(\frac{t-s}{\eta}\right)\,,$$
so that $\Psi$ converges to a Dirac in space and time  when $\lambda, \eta \to 0$;  and we choose  $\Phi$ such that $\Phi\rightarrow\mathbf{1}_{[0,T]\times B(x_0, R+M(T-t))}$ when $\theta$, $\epsilon\to 0$.
The last step consists in estimating the integral (\ref{eq:sumtv}) when $\eta, \theta, \epsilon\to 0$, keeping the $\lambda$ fixed to make appear $\mathcal{F}$.

\textbf{Estimate on $\mathcal{F}$.}
By the doubling variable method, we obtain
\begin{align}
\pt_T \mathcal{F}(T,\lambda)\leq& \pt_T\mathcal{F}(0,\lambda)+C\lambda \left(\pt_\lambda \mathcal{F}(T,\lambda)+\frac{N}{\lambda} \mathcal{F}(T, \lambda)\right) + C' \mathcal{F}(T,\lambda) +\lambda\int_0^T A(t)\mathrm{d}t\,,\label{eq:1}
\end{align}
where
\begin{align*}
 A(t)&=M_1\int_{\reali^N}\norma{\nabla(F-\div\, f)(t,x,\cdot)}_{\L\infty(u)} \d{x},&M_1&=\int_{\reali^N} \norma{y}\mu(\norma{y})\mathrm{d}{y},\\
  C' &= N\norma{\nabla\pt_u f}_{\L\infty} +\norma{\pt_u F}_{\L\infty}\,. &C&=\norma{\nabla\pt_u f}_{\L\infty}\,.
\end{align*}

We use  the hypothesis $u_0\in \BV(\reali^N, \reali)$ to get $\frac{1}{\lambda}\pt_T \mathcal{F}(0,\lambda)\leq M_1 \tv(u_0)$, and we integrate with respect to time, obtaining:
\begin{equation*}
0\leq \frac{M_1 }{C}\tv(u_0) +\pt_\lambda \mathcal{F}(T,\lambda) +\frac{\alpha(T)}{\lambda}\mathcal{F}(T,\lambda) +\frac{1}{C}\int_0^TA(t)\mathrm{d}t\,,
\end{equation*}
where
$$\alpha(T)= N+\frac{C'}{C}-\frac{1}{C\,T}\rightarrow_{T\to 0} -\infty\,.$$  
Next, we choose $T$  so that $\alpha<-1$ and we integrate on $[\lambda,+\infty[$. 
We obtain:
\[
\mathcal{F}(T,\lambda)\leq \frac{\lambda}{-\alpha-1} \frac{M_1}{C} \,\tv(u_0) +\frac{\lambda}{C(-\alpha-1)}\int_0^T A(t)\mathrm{d}t\,.
\]
Besides, note that, thanks to the shape of $\mu$, we can find $K>0$ so that
\begin{equation*}
\pt_\lambda\mathcal{F}(T,\lambda) +\frac{N}{\lambda}\mathcal{F}=\int_0^T\int_{\reali^N}\int_{B(x_0,R+M(T_0-t))}  \modulo{u(x+y)-u(x)}\mu'_{\lambda}(y)\mathrm{d}{x}\mathrm{d}{y}\mathrm{d}t \;\leq\; \quad\frac{K}{2\lambda} \mathcal{F}(T,2\lambda)\,. 
\end{equation*}
Hence, we finally obtain
\begin{eqnarray*}
\frac{1}{\lambda} \pt_T \mathcal{F}(T,\lambda) &
\leq & \frac{1}{\lambda} \pt_T\mathcal{F}(0,\lambda) +  \frac{CK+C'}{(-\alpha-1)C}  \left(M_1\tv(u_0) +\int_0^T A(t)\mathrm{d}t\right) +\int_0^T A(t)\mathrm{d}t\,.
\end{eqnarray*}
By Proposition \ref{prop:tv} we have then $u(t)\in \BV(\reali^N, \reali)$.We can now divide (\ref{eq:1})  by $\lambda$ and make $\lambda\to 0$. Applying a standard Gronwall-type argument, we obtain the desired inequality.
\end{proof}

\subsection{Scalar conservation law with a non-local flow}
Let us consider equation (\ref{eq:panic}) with initial condition $\rho_0\in (\L1\cap\L\infty\cap\BV)(\reali^N, \reali)$. We want to prove 
Theorem \ref{thm:panicK} and Theorem \ref{thm:Gdiff}. Note that this proof adapt smoothly to prove Theorem \ref{thm:order} and \ref{thm:panicK2}. 
We use here the scheme proposed in Section \ref{sec:intro}, choosing $X=\C0([0,T], \L1(\reali^N, \reali))$ at point \textbf{(a)}. Point  \textbf{(b)} is then satisfied by Kru\v zkov theorem (see Theorem  \ref{thm:kru}) and point \textbf{(c)} is satisfied by the stability estimate given by Theorem \ref{thm:stab} \& \ref{teo:tv}.

This scheme of proof can be also used for the other models: the proof of Theorem \ref{thm:order} for an orderly crowd is based on it, as well as Theorem \ref{thm:panicK2} giving existence and uniqueness for the several populations models (\ref{eq:panic2}) and (\ref{eq:order2}). For more details on these various theorems, refer to \cite{ColomboLecureux, ColomboGaravelloMercier}.

We use also this scheme of proof in order to prove existence and uniqueness of solution to equation (\ref{eq:groupindiv}), where we need furthermore stability results on the ODE (see \cite{ColomboMercier}).

\subsubsection{Proof of Theorem \ref{thm:panicK}}\label{sec:proofK}
%\textbf{Lemme : } Le probl\`eme de Cauchy $\pt_t r +\Div(r w(t,x))=R$ admet une unique solution  entropique de Kru\v zkov,     

%Let $ \beta>\alpha>0$ and $T\leq T_*=\frac{\ln(\beta/\alpha)}{C(\beta)}$. 
Let us introduce the space
$$X=\C0([0,T],\L1(\reali^N;\rpic))\,.$$ 
Let $r\in X$  
and $\rho\in X$ be the solution of  (\ref{eq:fix}) given by Kru\v zkov Theorem (see Theorem \ref{thm:kru}). Note that, $u=0$ being a solution to (\ref{eq:fix}),  $\rho_0\geq 0$ implies $\rho(t)\geq 0$ for all $t\geq 0$ thanks to the maximal principle on scalar conservation laws \cite[Theorem 3]{Kruzkov}.

\begin{remark}
For the model (\ref{eq:order}) of orderly crowd, we can use the same king of argument to have a uniform bound on the $\L\infty$ norm. Indeed, if $u=1$ is a solution to the equation 
\begin{equation}\label{eq:borne1}
\pt_t \rho +\Div (\rho v(\rho) \vec w(x))=0
\end{equation}
 then $\rho_0\leq 1$ implies $\rho(t)\leq 1$ for all $t\geq 0$. This fact allows us in this case  to consider the space $X=\C0([0,T], \L1(\reali^N, [0,1]) )$.
The same fact is also true for each density of the model (\ref{eq:order2}) of orderly crowd with several populations. Indeed,  the interaction between the equations is only in the non-local term; so, if we fix the nonlocal term, we have to deal with the same kind of scalar conservation law as (\ref{eq:borne1}). In particular, for each density $\rho_i$ the maximum principle gives us 
$$0\leq \rho_{0,i}\leq 1 \qquad\Rightarrow \qquad 0\leq \rho_i(t)\leq 1 \quad \textrm{ for all }t\geq 0\,.$$
So finally, if there are $k$ populations, we have the following uniform bound on the total density  density 
$$
0\leq \sum_{i=1}^k \rho_i(t)\leq k\,.
$$
\end{remark}
%
%
%The choice of $T_*$ gives
%$$\norma{\rho(T)}_{\L\infty}=\norma{\rho_0(X)\, \exp({\int_0^T\div V(w)(t,X) \mathrm{d}t}) }_{\L\infty}\leq \beta\,.$$
%

Let us  introduce the application
$$\mathscr{Q}  : r \in X\rightarrow \rho\in X\,.$$
For  $r_1, r_2$, we get by Theorem \ref{thm:stab} and Theorem \ref{teo:tv},
\[
\norma{\mathscr{Q}(r_1)-\mathscr{Q}(r_2)}_{\L\infty([0,T],\L1)} \leq f(T)\norma{r_1-r_2}_{\L\infty([0,T],\L1)}\,,
\]
where $f$  is  continuous, increasing, $f(0)=0$ and $f\to_{T\to \infty}\infty$. Consequently, we can choose $T$ small enough so that $f(T)=1/2$ and apply the Banach fixed point theorem for a small time. 
Iterating the process in time, we obtain global in time existence.

\subsubsection{G\^ateaux derivative of the semi-group}\label{sec:proofGdiff}

We prove now that the semi-group obtained in Theorem \ref{thm:panicK} (or \ref{thm:panicK2}) is Gâteaux-differentiable with respect to the initial condition. The following result is obtained by the use of Kru\v zkov theory. 

Note that the Gâteaux-differentiability is also  possible for the more non-linear models  of orderly crowd (\ref{eq:order}) and (\ref{eq:order2}). Indeed an important tool below is that the equation (\ref{eq:panic}) preserves the regularity of the initial condition, which is no longer true for (\ref{eq:order}).

\begin{definition}
We say the application $S:\L1(\reali^N;\reali)\to \L1(\reali^N;\reali)$ is \emph{G\^ateaux differentiable  in $\L1$,  in $\rho_0\in \L1$, in the  direction $r_0\in \L1$}  if there exists a linear continuous application   $DS(\rho_0):\L1\to \L1$ such that
\[
\norma{ \frac{S(\rho_0+hr_0)-S(\rho_0)}{h}-DS(\rho_0)(r_0)}_{\L1}\to_{h\to 0} 0\,.
\] 
\end{definition}

We expect the Gâteaux derivative to be the solution of the linearized problem (\ref{eq:linearised}).
%\[
%\pt_t r +\Div(rV(\rho)+\rho DV(\rho)(r))=0\,,\quad r(0,\cdot)=r_0\,.
%\] 
%

First, we prove that the linearized problem admits an entropic solution (see \cite{ColomboHertyMercier, ColomboLecureux}).
%On montre  que le probl\`eme lin\'earis\'e admet une unique solution entropique :
\begin{theorem}\label{thm:existlin}  
Assume that  $v\in (\C4\cap\W2\infty)(\reali ,\reali)$, $\vec \nu \in (\C3\cap\W21\cap\W2\infty)(\reali^N, \reali^N)$, $\eta \in (\C3\cap\W2\infty)(\reali^N, \rpic)$. Let $\rho\in \C0([0,T_{ex}[;\W1\infty\cap\W11(\reali^N, \reali))$, $r_0\in (\L1\cap\L\infty)(\reali^N;\reali)$. 
Then the linearized problem (\ref{eq:linearised}) with initial condition $r_0$ admits a unique entropic solution $ r\in \C0([0,T_{ex}[;\L1(\reali^N;\reali))$ and we denote $\Sigma_t^{\rho} r_0 =r(t,\cdot)$. 

If furthermore $r_0\in \W11(\reali^N, \reali)$, then for all $t\in [0,T_{ex}[$, $r(t)\in \W11(\reali^N, \reali)$.
\end{theorem}
The proof is similar to the one of Theorem \ref{thm:panicK} so we omit it.
 
\begin{theorem}
Assume that  $v\in (\C4\cap\W2\infty)(\reali ,\reali)$, $\vec \nu \in (\C3\cap\W21\cap\W2\infty)(\reali^N, \reali^N)$, $\eta \in (\C3\cap\W2\infty)(\reali^N, \rpic)$. Let $\rho_0\in (\W1\infty\cap\W21)(\reali^N, \reali^+)$, $r_0\in (\W11\cap\L\infty)(\reali^N, \reali)$.

Then, for every $t\geq 0$, the local semigroup of the pedestrian traffic problem is $\L1$ G\^ateaux differentiable in the direction $r_0$ and
\[
 DS_t(\rho_0)(r_0)=\Sigma_t^{S_t \rho_0} r_0\,.
\]
\end{theorem}
We prove this theorem following the same sketch of proof as above.

\begin{proof}
Let $\rho, \rho_h$ the solutions of the problem  (\ref{eq:panic}) with initial condition $\rho_0, \rho_0+hr_0$. 

Let $r$ be the solution of the linearized problem (\ref{eq:linearised}), $r(0)=r_0$ and let $z_h=\rho+hr$ which then satisifies
\[
\pt_t z_h+\Div\left(z_h(v(\rho*\eta)+hv'(\rho*\eta )(r*\eta))\vec \nu(x) \right) =h^2\Div(rv'(\rho*\eta)(r*\eta)\vec\nu(x) ) \,,\quad z_h(0)=\rho_0+h r_0\,.
\]    

Then we use Theorem \ref{thm:stab} \& \ref{teo:tv}  to compare $\rho_h$ et $z_h$. We get:
\begin{align*}
\frac{1}{h}\norma{\rho_h - z_h}_{\L\infty([0,T[,\L1)}\leq& F(T)\left(\frac{1}{h}\norma{\rho_h- \rho}_{\L\infty(\L1)}^2 +\frac{1}{h}\norma{\rho_h-z_h}_{\L\infty(\L1)} \right)\\
&+h C(\beta) T e^{C(\beta) T}\norma{r}_{\L\infty(\W11)}\norma{r}_{\L\infty(\L1)}\,,
\end{align*}
with $F$ increasing and $F(0)=0$. Then, we choose $T$ small enough so that $F(T)=1/2$. 

Besides, applying Theorem \ref{thm:stab} to compare $\rho$ and $\rho_h$ and using Gronwall lemma, we obtain that $\frac{1}{h}\norma{\rho_h-\rho}_{\L\infty(\L1)}$ remains bounded when $h\to 0$. 

Finally, we can take the limit $h\to 0$ and obtain the desired result.
\end{proof}

\subsubsection{Extrema of a cost functional}
We are now able to characterize the extrema of a given cost functional.
Let $J$ be a cost functional such that
\[
    J (\rho_0 )
    =
    \int_{\reali^N}
    f \left( S_t \rho_0 \right) \, \psi(t,x)
    \mathrm{d}{x} \,.
  \]

\begin{proposition}
  Let $f\in \C{1,1}(\reali;\rpic)$ and $\psi \in
  \L\infty(\reali^+\times \reali^N; \reali)$. Assume  $S
  \colon [0,T] \times (\L1 \cap \L\infty) (\reali^N;\reali) \to (\L1 \cap
  \L\infty) (\reali^N;\reali)$  is  $\L1$ G\^ateaux
  differentiable.
  
If $\rho_0 \in (\L1
  \cap \L\infty) (\reali^N;\reali)$  is solution of the problem:
  \begin{displaymath}
    \textrm{Find } \min_{\rho_0} {J}( \rho_0  )
    \mbox{ such that}  \left\{ S_t\rho_0
    \textrm{ is solution of problem (Traffic)}\right\}.
  \end{displaymath}
  then, for all $r_0 \in (\L1 \cap \L\infty) (\reali^N; \reali)$
  \[
    \int_{\reali^N} f'(S_t \rho_0) \, \Sigma_t^{\rho_0} r_0 \, \psi(t,x)
    \, \mathrm{d}x
    =
    0 \,.
  \]
\end{proposition}

\section{Using the Optimal transport theory}\label{sec:ot}

In this section, we want to prove Theorem \ref{thm:panicOT}, giving existence and uniqueness of weak measure solurtion to (\ref{eq:panic}) and (\ref{eq:panic2}). In a more general settings, we consider the system: 
\begin{align}
\pt_t \rho_i +\div (\rho_i V_i(x, \rho_{1}*\eta_{i,1}, \ldots, \rho_k*\eta_{i,k}))&=0\,,&(t, x)\in&\rpic\times\reali^N  \label{eq:system}\\
\rho_i(0)&=\bar\rho_i\,,& i\in &\{ 1,\ldots , k \}\,.\nonumber
\end{align}
in which the  coupling between the equations is only present through the nonlocal term.
Fixing the nonlocal term, we obtain a system of decoupled continuity equations:
\[
\left\{
\begin{array}{l}
\pt_t \rho_1 +\div (\rho_1\,b_1(t,x))=0\,,\\
\ldots\\
\pt_t \rho_k +\div (\rho_k\,b_k(t,x))=0\,,
\end{array}
\right.
\]
where $b_1, \ldots, b_k$ are regular with respect to $x$.

\subsection{Existence and uniqueness of weak measure solutions}

Let us remind the following definitions of weak measure solution, but also of Lagrangian solution. 
\begin{definition}
$\rho\in \L\infty([0,T], \mathcal{M}^+(\reali^N)^k)$ is a  {measure solution} of (\ref{eq:system}) if, $\forall \varphi \in \Cc\infty (]-\infty, T]\times\reali^N, \reali)$
\[
\int_0^T \int_{\reali^N} \big[ \pt_t\phi + V_i(x, \rho* \eta_i)\cdot\nabla\phi \big] \d{\rho^i_t (x)}\d{t} +\int_{\reali^N} \phi(0,x)\d{\bar \rho_i(x)}=0\,.
\]
\end{definition}

\begin{definition} $\rho\in \L\infty( [0,T], \Mes^+(\reali^N)^k)$ is a {Lagrangian solution}  with initial condition $\bar \rho\in \Mes^+(\reali^N)^k$ if there exists an ODE flow $X^i: [0,T]\times \reali^N \to \reali^N$, solution of 
\[
\left\{
\begin{array}{rl}
\dst\frac{\d{ X^i}}{\d{t}}(t,x)=&V_i(X^i(t,x), \rho_t* \eta^i(X^i(t,x)))\,,\\
\dst X^i(0,x)=&x\,;
\end{array}
\right.
\]
and such that $\rho^i_t={X^i_t}_\sharp \bar\rho^i$ where $X^i_t:\reali^N\to\reali^N$ is the map defined as $X^i_t(x)=X^i(t,x)$ for any $(t,x)\in \rpic\times\reali^N$. 
\end{definition}

\begin{definition}
Let $\mu $ be measure on $\Omega$ and $T:\Omega\to \Omega'$ a measurable map. Then $T_\sharp \mu$ is the {push-forward} of $\mu$ if  for any $\phi\in \Cc0(\Omega')$, 
$$\int_{\Omega'}\phi(x)\d{T_\sharp\mu(x)}=\int_{\Omega}\phi\left( T(y)\right) \d{\mu(y)}\,.$$  

If $T_\sharp\d\mu(x)=f(x)\d{x}$ and $\d\mu(y)=g(y)\d{y}$; and $T$ is a $\C1$-diffeomorphism. Then we have the \emph{change of variable formula}
$$g(x)=f(T(x))\modulo{\det(\nabla T(x))}\,.$$
\end{definition}

Instead of proving directly Theorem \ref{thm:panicOT} (or Theorem \ref{thm:panicOT2}), we prove first the existence and uniqueness of lagrangian solutions. 
\begin{theorem}\label{thm:main}
Let $\bar\rho \in \Mes^+(\reali^N,\reali^k)$. Let us assume that $V\in (\L\infty\cap\lip)(\reali^N\times\reali^k,
\reali^{N\times k})$ and that $\eta\in(\L\infty\cap\lip)(\reali^N, \reali^{k\times k})$. Then {there exists a unique  Lagrangian solution} to  system (\ref{eq:system}) with initial condition $\bar \rho$.
%
% Furthermore, this measure solution is also a {Lagrangian solution}.
\end{theorem}

\begin{remark}
Note that a Lagrangian solution is also a weak measure solution. 

Besides, we recover the results of Theorem \ref{thm:panicK}, except the continuity in time. Indeed, Lagrangian solutions conserve the regularity of the initial condition along the time:
if $\bar \rho\in \L1(\reali^N,(\reali^+)^k)$ then $\rho\in \L\infty(\reali^+, \L1(\reali^N, (\reali^+)^k))$  and we have $$\norma{\rho(t)}_{\L1}=\norma{\bar\rho}_{\L1}.$$

\vspace{0.2cm}
If moreover $\bar \rho \in (\L1\cap\L\infty)(\reali^N,(\reali^+)^k)$, then $\rho(t)\in \L\infty$ for all $t\geq 0$ and we have the estimate 
\[
\norma{\rho(t)}_{\L\infty}\leq \norma{\bar\rho}_{\L\infty}e^{Ct}\,,
\]
where $C$ depends on $\norma{\bar\rho}_{\Mes}$,  $V$ and $\eta$.
\end{remark}

We conclude the proof of Theorem \ref{thm:panicOT} stating the uniqueness of measure solutions, so that the Lagrangian solution obtained is also the unique measure solutions.
\begin{proposition}\label{prop:uni}
Under the same set of hypotheses as Theorem \ref{thm:main}, the weak measure solutions for (\ref{eq:system}) are unique.
\end{proposition}

\subsection{Proof of Theorem \ref{thm:main}} \label{sec:proofot}
Once again, the proof is based on the scheme described in Section \ref{sec:intro}. We consider below probability measures instead ; for any $i$, we could also consider positive measures of fixed total mass $\norma{\bar\rho^i}_{\mathcal{M}}$.

\noindent\textbf{(a)} Let us first introduce  a space 
$$ X=\L\infty([0,T], \mathcal{P}(\reali^N)^k)\,.
$$
We equip this space with the distance $d(\mu, \nu)=\sup_{t\in [0,T]} \mathcal{W}_1(\mu_t, \nu_t)$. The Wasserstein distance $\mathcal{W}_1$ is defined as follow:
\begin{definition}
Let $\mu$, $\nu\in \mathcal{P}(\reali^N)$. Let us denote $\mathbb{P}_x:\reali^d\times\reali^d\to \reali^d$ the \emph{projection on the first coordinate}; that is, for any $(u,v)\in \reali^d\times\reali^d$,  $\mathbb{P}_x(u,v)=u$. In a similar way, $\mathbb{P}_y:\reali^d\times\reali^d\to \reali^d$ is the \emph{projection on the second coordinate}; that is, for any $(u,v)\in \reali^d\times\reali^d$,  $\mathbb{P}_y(u,v)=v$. We denote $\Xi\, (\mu, \nu)$  the set of \emph{plans}, that is
$$\Xi(\mu, \nu)=\left\{\gamma\in \mathcal{P}(\reali^N\times \reali^N) : {\mathbb P_x}_\sharp \gamma=\mu \textrm{ and }{\mathbb P_y}_\sharp \gamma=\nu\right\}.$$ 

The \emph{Wasserstein distance of order one} between $\mu$ and $\nu$  is
\[
W_1(\mu, \nu)= \inf_{\gamma \in \,\Xi\, (\mu, \nu)} \int_{\reali^N\times\reali^N} \modulo{x-y}\d{\gamma (x,y)}\,.
\]

Let $\rho=(\rho^1, \ldots , \rho^k)$, $\sigma=(\sigma^1, \ldots , \sigma^k)\in \mathcal{P}(\reali^N)^k$. The \emph{Wasserstein distance of order one} between $\rho$ and $\sigma$, denoted $\mathcal{W}_1(\rho, \sigma)$, as
\[
\mathcal{W}_1(\rho, \sigma)=\sum_{i=1}^k W_1(\rho^i, \sigma^i)\,.
\]
\end{definition}

Let us recall the following duality formula:
\begin{proposition}[cf. Villani {\cite[p. 207]{Villani}}] \label{prop:villani}
Let $f,g$ be two probability measures. The Wasserstein distance of order one between $f$ and $g$ satisfies
\[
W_1(f,g)=\sup_{\lip (\phi)\leq 1 } \int_{\reali^d}\phi(x)\left(\d{f(x)}-\d{g(x)}\right)\,.
\]
\end{proposition}

\noindent\textbf{(b)} Existence of Lagrangian solutions. 

Let $r\in \L\infty([0,T], \mathcal{P}(\reali^N)^k)$. \\
Define $b_i(t,x)=V_i(x, r_t*\eta_i)\in\L\infty([0,T], \W1\infty(\reali^N)^k)$. Let us consider the equation
\begin{equation}\label{eq:cont}
\pt_t\rho_i+\div \left(\rho_i  b_i(t,x)\right)=0.
\end{equation}
Let $\rho_i={X_t}_\sharp\bar\rho$ be the Lagrangian solution of (\ref{eq:cont}). Then the application
 $$\mathscr{T}:r\in \L\infty([0,T], \mathcal{P}(\reali^N)^k) \mapsto \rho\in \L\infty([0,T], \mathcal{P}(\reali^N)^k)\,.$$
is well-defined.

\noindent\textbf{(c)} Stability estimate.
\begin{proposition}
Let $\bar \rho, \bar \sigma\in \mathcal{P}(\reali^N)$ and $r, s\in \C0([0,T],\mathcal{P}(\reali^N))$.
Let  $V \in (\L\infty\cap\lip)(\reali^N\times \reali^k,\reali^N)$, $\eta, \nu\in (\L\infty\cap\lip)(\reali^N,\reali)$. If $\rho$ and $\sigma$ are Lagrangian solutions of 
\[
\begin{array}{l}
\pt_t \rho +\div (\rho \,V( x, r* \eta ))=0\,,\qquad \rho (0,\cdot)=\bar\rho \,,\\
\pt_t \sigma +\div (\sigma\,  V( x, s* \eta ))=0\,,\qquad \sigma(0,\cdot)=\bar\sigma\,.
\end{array}
\]
We have the estimate:
\[
\mathcal{W}_1(\rho_T, \sigma_T)\leq e^{CT} \mathcal{W}_1(\bar\rho, \bar \sigma)+ T \,e^{CT} \,C' \,\sup_{t\in [0,T]} \mathcal{W}_1(r_t, s_t) \,,
\]
where $C=\lip_x(V)+\lip_r (V)\lip(\eta) \norma{\bar \rho}_{\Mes}+ \lip_r(V)\lip (\eta) \norma{\bar\rho}_{\Mes}$ and $C'=\lip_r(V)\lip(\eta)\norma{\bar\rho}_{\Mes}$.
\end{proposition}

\begin{proof}
Let $X, Y$ be the ODE flows associated to $\rho, \sigma$. Let $\gamma_0\in \Xi(\bar\rho, \bar \sigma)$. 
Define
\[
X_t\Join Y_t : (x,y)\mapsto (X(x), Y(y))\,.
\]
Then $\gamma_t=(X_t\Join Y_t )_\sharp \gamma_0\in \Xi(\rho_t, \sigma_t)$. Let us  introduce
\[
Q(t)=\int_{\reali^N\times\reali^N} \modulo{x-y}\d{\gamma_t(x,y)}=\int_{\reali^N\times\reali^N}\modulo{X_t(x)-Y_t(y)}\d{\gamma_0(x,y)}\,.
\]
Then $Q$ is a Lipschitz function and   
\[
Q'(t)\leq \int_{\reali^N\times\reali^N} \modulo{V(X_t(x), r_t*\eta(X_t(x))) -V(Y_t(y), s_t*\eta(Y_t(x))) }\d{\gamma_0(x,y)}.
\]
By triangular inequality, we obtain
\begin{align*}
Q'(t)\leq  &  (\lip_x(V)+\lip_r(V)\lip(r_t*\eta)) Q(t)+\lip_r(V)\int_{\reali^N\times\reali^N} \modulo{(r_t-s_t)*\eta(Y_t(y))}\d{\gamma_0(x,y)}\,.
\end{align*}
Note besides that, thanks to Proposition \ref{prop:villani}, we have
\[
{(r^i_t-s^i_t)*\eta(z)}=\int_{\reali^N}\eta(z-\zeta)(\d{r^i_t(\zeta)}-d{s^i_t(\zeta)})\leq \lip(\eta)W_1(r_t^i,s_t^i)\,.
\]
Integrating, we get 
\[
Q(t)\leq Q(0)e^{Ct} +t C'\,e^{C t}\sup_\tau \mathcal{W}_1(r_t, s_t)\,.
\]
where $C=\lip_x(V)+\lip_r(V)\norma{r_t}_{\Mes} \lip(\eta)$, $C'=\lip_r(V)\lip(\eta) \norma{\bar\rho}_{\Mes}$.

We conclude taking $\gamma_0$ in an optimal way so that $Q(0)=\mathcal{W}_1(\rho_0, \sigma_0)$ and using the inequality $$\mathcal{W}_1(\rho_t, \sigma_t)\leq Q(t)\,.$$
\end{proof}

The stability estimate allows us to apply Banach fixed point Theorem for $T$ small enough.

\subsection{Measure solutions are Lagrangian solutions}
\begin{proofof}{Proposition \ref{prop:uni}}
Let $\rho$ be a measure solution of (\ref{eq:system}). Let $b=V(x, \rho*\eta)$ and
denote $\sigma$ the Lagrangian solution associated to $\pt_t\sigma+\div (\sigma b)=0$ with $\sigma(0)=\bar \rho$.

Then $\delta=\rho-\sigma$ is a measure solution of 
$
\pt_t \delta+\div (\delta b)=0$, with $\delta(0)=0$. 
That is to say, for any $\phi\in \Cc\infty(]-\infty, T]\times \reali^N, \reali)$, 
$$
\int_0^T\int_{\reali^N} \left(\pt_t \phi+b^i(t,x) \cdot\nabla\phi\right)\d{\delta_t}\d{t}=0\,.
$$
Let $\psi \in \Cc0(]-\infty, T]\times\reali^N, \reali)$.  We can find $\phi\in \Cc1(]-\infty, T]\times\reali^N, \reali )$ so that $\psi=\pt_t\phi+b^i(t,x)\cdot \nabla \phi$. Hence, for any $\psi\in \Cc0(]-\infty, T]\times\reali^N, \reali)$, we have $\int_0^T \int_{\reali^N}\psi \d{\delta_t}\d{t}=0$, which implies $\delta\equiv 0$ a.e.  so  $\rho=\sigma$ a.e..
Consequently, we have  $b^i(t,x)=V^i(x, \sigma * \eta^i)$, and $\rho$ is a Lagrangian solution of (\ref{eq:system}). 
\end{proofof}

\section{Conclusion}
%%%%%%%%%%%%%%%%%%%%%

In the Kru\v zkov framework, we are able to prove existence and uniqueness of weak entropy solution for 
 the equation $\pt_t \rho+\div (\rho V(x, {\rho}, \rho*\eta))=0.$
Furthermore, we can prove  uniform bound in $\L\infty$ if $V=v(\rho)\vec W(x, \rho*\eta)$, with $v(1=)0$. However, the required hypotheses are very strong: we need indeed $V\in \C2\cap\W21\cap\W2\infty$. 

In the optimal transport theory framework, we can treat only equations such that
$\pt_t \rho+\div (\rho V(x, \rho*\eta))=0\,.$
For this equation, we have only $\L\infty$ bound that are exponentially growing in time. The hypotheses are nevertheless weaker since we only ask $V\in \Lip\cap\L\infty$, but we are no longer able to prove the Gâteaux-differentiability of the semi-group.

{\small

  \bibliography{procCemracs}

  \bibliographystyle{abbrv}
}

\end{document}